\pdfmapfile{=<name>.map}
\documentclass[12pt,oneside,reqno,fleqn]{amsart}
\usepackage{amssymb}
\usepackage{bbm}
\usepackage{cases}
\usepackage{amsmath,mathtools}
\usepackage{graphicx}
\usepackage{mathrsfs}
\usepackage{stmaryrd}
\usepackage{color}
\usepackage{soul}
\usepackage[dvipsnames]{xcolor}
\usepackage{amsfonts}
\usepackage{dsfont}
\usepackage{amsmath,amssymb,amsthm}
\usepackage{libertine}
\usepackage[T1]{fontenc}
\usepackage[libertine,varbb]{newtxmath}
\usepackage[colorlinks,linkcolor=Red,citecolor=Red,urlcolor=Red]{hyperref}
\usepackage[shortlabels] {enumitem}
\usepackage[nameinlink,capitalize]{cleveref}
\usepackage[spacing=true,kerning=true,tracking=true,letterspace=50]{microtype}
\usepackage{orcidlink}

\usepackage{mhequ}
\usepackage{xcolor}
\colorlet{darkblue}{blue!90!black}
\colorlet{darkred}{red!90!black}
\definecolor{Green}{rgb}{0.01, 0.75, 0.24}

\newcommand{\dd}{\,\mathrm{d}}
\usepackage{todonotes}
\hypersetup{pdftitle={Milstein}, pdfauthor={ M\'at\'e Gerencs\'er, Gerald Lampl and Chengcheng Ling}}
\numberwithin{equation}{section}
\newcommand{\be}{\begin{eqnarray}}
\newcommand{\ee}{\end{eqnarray}}
\newcommand{\ce}{\begin{eqnarray*}}
\newcommand{\de}{\end{eqnarray*}}
\newtheorem{theorem}{Theorem}[section]
\newtheorem{lemma}[theorem]{Lemma}
\newtheorem{proposition}[theorem]{Proposition}
\newtheorem{corollary}[theorem]{Corollary}
\theoremstyle{remark}
\newtheorem{assumption}[theorem]{Assumption}

\newtheorem{example}[theorem]{Example}
\newtheorem{remark}[theorem]{Remark}
\newtheorem{definition}[theorem]{Definition}

\crefname{eqn}{Equation}{Equations}
\crefname{assumption}{Assumption}{Assumptions}
\crefname{innercustomthm}{Condition}{Conditions}
\crefrangelabelformat{innercustomthm}{#3#1#4-#5#2#6}

\def\<{{\langle}}
\def\>{{\rangle}}
\def\({{\Big(}}
\def\){{\Big)}}

\def\bx{{\mathbf{x}}}

\def\={&\!\!=\!\!&}
\def\bt{\begin{theorem}}
\def\et{\end{theorem}}
\def\bl{\begin{lemma}}
\def\el{\end{lemma}}
\def\br{\begin{remark}}
\def\er{\end{remark}}
\def\bd{\begin{definition}}
\def\ed{\end{definition}}
\def\bp{\begin{proposition}}
\def\ep{\end{proposition}}
\def\bc{\begin{corollary}}
\def\ec{\end{corollary}}
\def\bx{\begin{example}}
\def\ex{\end{example}}

\def\cD{{\mathcal D}}

\def\cF{{\mathcal F}}

\def\cH{{\mathcal H}}
\def\chh{\cH}

\def\cM{{\mathcal M}}

\def\cS{{\mathcal S}}

\def\mB{{\mathbb B}}

\def\mD{{\mathbb D}}
\def\mE{{\mathbb E}}
\def\E{\mE}

\def\mN{{\mathbb N}}

\def\mP{{\mathbb P}}

\def\mR{{\mathbb R}}

\def\geq{\geqslant}
\def\leq{\leqslant}

\def\frD{{\mathfrak{D}}}
\newcommand{\R}{{\mathbb R}}
\newcommand{\N}{{\mathbb{N}}}

\newcommand{\norm}[1]{{\left\vert\kern-0.25ex\left\vert\kern-0.25ex\left\vert #1
    \right\vert\kern-0.25ex\right\vert\kern-0.25ex\right\vert}}
\DeclareMathOperator*{\esssup}{ess\,sup}

\allowdisplaybreaks
\begin{document}
	\title[SDEs]{The Milstein scheme for singular SDEs with H\"older continuous drift}

	\date{\today}
\author{M\'at\'e Gerencs\'er\,\orcidlink{0000-0002-7276-7054}, Gerald Lampl and Chengcheng Ling\,\orcidlink{0000-0002-0106-962X}}

\address{
Technische Universit\"at Wien, Wiedner Hauptstra\ss e 8-10,
 1040 Vienna, Austria}
\email{ \{mate.gerencser,gerald.lampl,chengcheng.ling\}@asc.tuwien.ac.at }

	\begin{abstract}
We study the $L^p$ rate of convergence of the Milstein scheme for SDEs when the drift coefficients possess only H\"older regularity. If the diffusion is elliptic and sufficiently regular, we obtain rates consistent with the additive case. The proof relies on regularisation by noise techniques, particularly stochastic sewing, which in turn requires (at least asymptotically) sharp estimates on the law of the Milstein scheme, which may be of independent interest.
		
		\bigskip
		
		\noindent {{\sc Mathematics Subject Classification (2020): Primary 60H35, 60H10; Secondary 60H50,
60L90, 35B65.}
		
		}

	\end{abstract}
	
	\maketitle
\section{Introduction}
The term \emph{regularisation by noise} classically refers
to well-posedness of stochastic differential equations (SDEs)
\begin{align}
    \label{eq:SDE}
    dX_t=b(X_t)\dd t+\sigma(X_t)\dd W_t,\quad X_0=x_0,
\end{align}
beyond the (stochastic version of the) Cauchy-Lipschitz theorem: at the price of some nondegeneracy assumption on $\sigma$, the classical Lipschitz continuity condition on $b$ can be dramatically reduced \cite{Zv, Ver,KR,Zhang2011}.
Recently numerous studies focused on leveraging these regularisation effects in the analysis of approximation of SDEs with irregular coefficients \cite{GR11, LS18,MGY20, GK21,  HRZ, JM21, BHY}.
From the long list of recent works \cite{GK, BDG,BDG1,DGL, HJK12, LL,MY-Milstein,MY-Lower,LS17,NSz,PSS,PT,Y-1,AGI,  LM10} and references therein,
let us highlight two features.
  On the one hand, when $b$ is merely bounded, then with nondegenerate and sufficiently regular $\sigma$, the Euler-Maruyama scheme is shown to converge in $L^p$ with rate $1/2$ in \cite{DGL}. On the other hand, if the noise is additive (i.e. $\sigma$ is constant and nondegenerate), then this can be improved: if in addition, $b$ has regularity $\alpha\in(0,1)$ in either a H\"older or a Sobolev sense (with sufficiently high integrability), $L^p$ rate $(1+\alpha)/2$ is proved \cite{BDG, DGL}.
A natural question is whether the rates beyond $1/2$ can also be achieved in the multiplicative case.
As far as weak convergence is concerned, this was affirmatively answered in \cite{Holland} in the case of H\"older $b$. For strong convergence, however, the rate $1/2$ is known to be sharp for the Euler-Maruyama scheme, see e.g. \cite{KPS94, JP, MT}, and therefore higher order methods are needed to even hope for a superior rate. The goal of this paper is to show that the standard Milstein scheme does achieve the $(1+\alpha)/2$ rate in the full range $\alpha\in(0,1)$ in the H\"older  drift case.

The Milstein scheme for \eqref{eq:SDE} is defined as
\begin{align}
   \label{eq:Milstein-scheme-SDE}
   dX_t^n=b(X_{k_n(t)}^n)\dd t+\Big(\sigma(X_{k_n(t)}^n)+\nabla\sigma\sigma(X_{k_n(t)}^n)(W_t-W_{k_n(t)})\Big)\dd W_t,\quad X_0^n=x_0^n,
\end{align}
with $k_n(t)=\frac{\lfloor nt\rfloor}{n}$, $n\in\mathbb{N}$, $t\in[0,1]$.
 This scheme was originally designed by Milstein \cite{Milstein} to produce a $O(n^{-1})$-error (in $L^p$) for SDEs with $C^2$ coefficients, just like the standard Euler scheme for deterministic ODEs. However in the framework of SDEs, as we see in \eqref{eq:Milstein-scheme-SDE}, such scheme is of second order compared to the Euler scheme.
This can pose challenges in implementation, see \cite[Section~7.5.2]{Pages}, but since these issues are well studied in the literature, we only focus on the error analysis of the scheme.
 \\

In the sequel $(X_t)_{t\in[0,1]}$ and $(X_t^n)_{t\in[0,1]}$ are $\R^d$-valued stochastic processes, $b:\R^d\to\R^d$, $\sigma:\R^d\to\R^{d\times d_1}$, and $W$ is a $d_1$-dimensional standard Brownian motion.
The dimensions $d$ and $d_1$ are in principle arbitrary, but the ellipticity condition will imply $d_1\geq d$ automatically.
The way products of matrices and higher order tensors are understood is always clear from the context, so we often omit indices. To illustrate this (and to make the scheme itself completely precise): the $(i,k)$-th coordinate of the $d\times d_1$ matrix $\nabla\sigma\sigma(X_{k_n(t)}^n)(W_t-W_{k_n(t)})$ is given by
\begin{equ}
\sum_{j=1}^{d_1}\sum_{\ell=1}^d\partial_j\sigma^{ik}\sigma^{j\ell}(X^{(n)}_{k_n(t)})\big(W^{\ell}_t-W^{\ell}_{k_n(t)}\big).
\end{equ}

We now state our assumptions and main result with assuming some standard notation; all of which is defined precisely in the part \emph{Notation} below.
One important notion is the nondegeneracy of the noise (though this can be relaxed on regions where the drift is regular, see \cite[Section~1.4]{DGL}): we say a matrix-valued function $A:\mR^d\mapsto\mR^{d\times d}$ is \emph{uniformly elliptic} if there exists a $\lambda>0$ such that for all $x,\xi\in\R^d$
\begin{align}\label{def-uni-ellip}
\lambda|\xi|^2\leq\langle A(x)\xi,\xi\rangle\leq \lambda^{-1}|\xi|^2.
\end{align}
  \begin{assumption}
\label{ass:main}
 For some $\alpha\in(0,1]$, assume
\begin{itemize}
    \item[($H^b$)]
      $b\in\mathcal{C}^\alpha$
      \item[($H^\sigma$)]
      $\sigma\sigma^*$ is uniformly elliptic  and $\sigma\in\mathcal{C}^{3}$.
\end{itemize}
\end{assumption}
It is well-known (see e.g. \cite{Ver}) that under  \cref{ass:main} a unique strong solution to \eqref{eq:SDE} exists.
Our main result is as follows.
\begin{theorem}\label{thm:main}
 Let $(X_{t})_{t\in[0,1]}$ and $(X_{t}^n)_{t\in[0,1]}$ be the solutions to \eqref{eq:SDE} and \eqref{eq:Milstein-scheme-SDE} correspondingly.  If \cref{ass:main} holds, then, for all $n\in\mN$, for any $p\geq 1$,
 for all $\epsilon>0$,  the bound
\begin{align}\label{est:Milstein-bounds-Holder}
  \big\Vert \sup_{t\in[0,1]}|X_t-X_t^n|\big\Vert_{L^p_\omega}\leq N|x_0-x_0^n|+Nn^{-\frac{1+\alpha}{2}+\epsilon}
\end{align}
holds, where the constant $N$ depends on $\Vert b\Vert_{\mathcal{C}^\alpha},\Vert\sigma\Vert_{\mathcal{C}^{3}},\alpha,p,d,d_1,\lambda$, $\epsilon$.
\end{theorem}

\begin{remark}\label{rem:early}
    As mentioned above, the rate $\frac{1+\alpha}{2}$ above agrees with the additive case \cite{BDG}.
    For additive noise, however, the same rate is obtained with $b$ having only Sobolev regularity: $b\in W^{\alpha,p}$ with $p{\geq}\max(2,d)$ \cite{DGL, NSz}. For the Milstein scheme the treatment of Sobolev drift is currently beyond the scope of the available density estimates in \cref{sec:Mal}, we leave this for future investigations.
\end{remark}

\subsection*{Related works}
The first analysis of a higher order scheme in the case of irregular drift is to our best knowledge \cite{MY-Milstein} (recently extended to SDEs with finite activity jumps \cite{PSS}). Therein the scalar $d=d_1=1$ case is considered, and the irregularity of the drift is in the form of jump discontinuities. This corresponds to $\alpha=1/2$, $p=2$ in the context of  \cref{rem:early} and therefore the strong rate $3/4$ proven in \cite{MY-Milstein} is consistent with $\frac{1+\alpha}{2}$. The regularity assumption on $\sigma$ is weaker in \cite{MY-Milstein} than ours, and the diffusion may also degenerate away from the irregularities. However, the scheme is less direct and involves the knowledge of a rather nontrivial function of the coefficients (denoted by $G^{-1}$ therein), which, if not available, may also need to be approximated, introducing further errors. Nevertheless, an interesting fact is that for this class of coefficients this rate $3/4$ is sharp among all approximation methods based on evaluations of $W$ on a deterministic grid \cite{MY-Lower}, but can be improved by using an adaptive algorithm \cite{Y-1}.

\subsection*{Main idea of the proof}
Let us briefly outline the strategy of the proof. The aim is to estimate the difference $ X_t-X_t^n$ which has the representation
 \begin{align*}
   X_t-X_t^n&=x_0-x_0^n+\int_0^t[b(X_r)-b(X_{k_n(r)}^n)]\dd r
    \\ &\quad+\int_0^t[\sigma(X_r)-\sigma(X_{k_n(r)}^n)-(\nabla\sigma\sigma)(X_{k_n(r)}^n)(W_r-W_{k_n(r)})]\dd W_r
    =:I_0+I_1+I_2.
 \end{align*}
What is perhaps not immediate on the first sight is that the novel difficulties with the Milstein scheme do not lie with $I_2$. Indeed, writing
  \begin{align*}
 I_2=&\int_0^t[\sigma(X_r)-\sigma(X_{r}^n)]\dd W_r
    \\&+\int_0^t[\sigma(X_{r}^n)-\sigma(X_{k_n(r)}^n)-(\nabla\sigma\sigma)(X_{k_n(r)}^n)(W_r-W_{k_n(r)})]\dd W_r=:I_{21}+I_{22},
  \end{align*}
the Milstein scheme is designed precisely so that $I_{22}$ is of order $n^{-1}$. The term $I_{21}$ can be treated easily by an appropriate version of Gronwall's lemma.
As for $I_1$, we write
 \begin{align*}
  I_1=&\int_0^t[b(X_r)-b(X_r^n)]\dd r+\int_0^t[b(X_r^n)-b(X_{k_n(r)}^n)]\dd r=:I_{11}+I_{22}.
 \end{align*}
The treatment of $I_{11}$ relies on a version \emph{Zvonkin's transformation} \cite{Zv, Ver}, more precisely in the form of an It\^o-Tanaka trick.
This transformation gives rise to further quantities similar to $I_{12}$,  which (and $I_{12}$ itself) are handled via \emph{stochastic sewing} \cite{Le}.
An important ingredient for stochastic sewing is the behavior of the law of the process, and the required estimates on the law of the Milstein approximation $X^n$ are significantly more challenging than in the case of the Euler-Maruyama scheme (in fact, some ``usual'' bounds like two-sided heat kernel estimates do not even hold).
These estimates (see Section \cref{sec:Mal}), which can be of independent interest, are derived via a Malliavin calculus toolbox.

\subsection*{Structure of the paper}

     Based on the above, the article is organised as follows: we start with \cref{sec:Mal-basic} on Malliavin calculus which gives us general criteria for deriving the density estimates for a process in an It\^o's integral form, following with \cref{sec:driftless-est} which gives the  estimates on the law of the Milstein scheme and some auxiliary  processes by applying the theorems from \cref{sec:Mal-basic}.  In \cref{sec:inte-est-Mil} we derive bounds on additive functionals of the Milstein scheme via stochastic sewing. In \cref{sec:main-proof} we combine these bounds with Zvonkin's transformation to conclude the  proof for the main theorem . In \cref{App:PDE} we include some estimates on PDEs which is used in the proof  in \cref{sec:main-proof}.

\subsection*{Notation}\label{sec:Notation}

For $k\in\mathbb{N}$, $f:\mathbb{R}^d\mapsto \mathbb{R}$, denote $\partial_k f(x):=\frac{\partial f(x)}{\partial x_k}$ for $x\in\mathbb{R}^d$ and   $\nabla f(x):=(\partial_if(x))_{1\leq i\leq d}$, the derivative is understood in the weak sense.
For vector-valued $f$ we use the same notation, and
$\nabla^k f$ is defined via $\nabla(\nabla ^{k-1}f)$ iteratively. For a multi-index $k=(k_1,\ldots,k_d)\in\mN^d$, denote $\partial^k f(x):=\frac{\partial^{|k|} f(x)}{\partial x_{k_1}\cdots\partial x_{k_d}}$. If $k=(0,\ldots,0)$, we use convention $\partial^kf=f$.
We denote by ${C}^\infty_0$ (${C}_p^\infty$, resp.) the set of all continuously infinitely differentiable functions that, along with all of their partial derivatives, are compactly supported (of polynomial growth, resp.).

For $\alpha\in(0,1]$,  we set $\mathcal{C}^\alpha(\mR^d)$ to be the space of continuous functions such that
\begin{align*}
  \Vert f\Vert_{\mathcal{C}^\alpha}:=[f]_{\mathcal{C}^\alpha}+\sup_{x\in\mR^d}|f(x)|:=\sup_{x,y\in\mR^d,x\neq y}\frac{|f(x)-f(y)|}{|x-y|^\alpha}+\sup_{x\in\mR^d}|f(x)|<\infty.
\end{align*}
Here, and often below, we write $\mathcal{C}^\alpha$ instead of $\mathcal{C}^\alpha(\mR^d)$ for simplicity.
For $\alpha\in(0,\infty)$, we define $\mathcal{C}^\alpha(\mR^d)$ the space of all functions $f$ defined on $\mR^d$ having bounded derivatives $\partial^k f$ for multi-indices $k\in\mN^d$ with $|k|\leq \alpha$ so that
\begin{align*}
  \Vert f\Vert_{\mathcal{C}^\alpha}&:=\|f\|_{\mathcal{C}^{ \lfloor\alpha\rfloor}}+[f]_{\mathcal{C}^\alpha}:=\sum_{|k|\leq \alpha}\sup_{x\in\mR^d}|\partial^kf(x)|+\sum_{\alpha-1\leq|k|<\alpha} [\partial^kf]_{\mathcal{C}^{\alpha-|k|}}<\infty.
\end{align*}
Note that the $\mathcal{C}^\alpha$-norm always includes the supremum of the function.
 We also denote the space of bounded measurable functions $\mathcal{C}^0(\mR^d)$  with the supremum norm.
To be noticed that the functions in $\mathcal{C}^0$ do not need to be continuous.

In the following we denote the conditional expectation w.r.t. the $\sigma$-algebras of the filtration $(\mathcal{F}_t)_{t\geq0}$ as $\mE_t(\cdot):=\mE(\cdot|\mathcal{F}_t), t\geq0$.

On finite dimensional vector spaces we always use the Euclidean norm.

In proofs, the notation $a\lesssim b$ abbreviates the existence of $C>0$ such that $a\leq C b$, such that moreover $C$ depends only on the parameters claimed in the corresponding statement. If the constant depends on any further parameter $c$, we incorporate it in the notation by writing $a\lesssim_c b$.

\section{Estimates on the law of $X^n$ and related processes}\label{sec:Mal}

\subsection{Preliminaries of Malliavin calculus}\label{sec:Mal-basic}
Let $H=L^2([0,1],\R^{d_1})$  with inner product $\langle \cdot, \cdot \rangle_H$ and for $h\in H$ let us use the shorthand $W(h)=\int_0^1 h_t\,dW_t$.
By $\mathcal{S}$ we denote the class of random variables $X$ for which there exists an $n \in \mathbb{N}$, vectors $h_1, \ldots, h_n \in H$ and a function $f \in C_p^\infty(\mathbb{R}^d)$ such that
\begin{align*}
X = f(W(h_1), \ldots, W(h_n)).
\end{align*}
We call the elements of $\mathcal{S}$ \textit{smooth random variables}. More generally,  let $V$ be a Hilbert space and denote by $\mathcal{S}_V$ the space of $V$-valued smooth random variables of the form
\begin{align*}
X = \sum_{j=1}^n X_j v_j,
\end{align*}
where $v_j \in V$, $X_j \in \mathcal{S}$.
Without loss of generality, we may assume that each there exist $h_1, \ldots, h_m \in H$ and function $f_1,\ldots,f_n \in C_p^\infty(\mathbb{R}^d)$ such that $X_j=f_j(W(h_1),\ldots,W(h_m))$.
The \emph{Malliavin derivative} of such a random variable $X\in \cS_V$ is the $H\otimes V$-valued variable $DX$, defined by
\begin{align}\label{chain-rule}
DX = \sum_{j=1}^n\sum_{i=1}^m \partial_i f_j(W(h_1), \ldots, W(h_m)) h_i\otimes v_j.
\end{align}
In the sequel any vector space $U$ is identified with $U\otimes \R$, in particular if $X\in\mathcal{S}$, then $D X\in\mathcal{S}_H$.
The \emph{$k$-th Malliavin derivative} can be defined recursively by the above. We then have that $D^k X$  is a $H^{\otimes k}\otimes V$-valued random variable.
From this point on we will only take $V$ to be a finite dimensional vector space, and so we drop it from the notation whenever it does not cause confusion (and one can simply understand every operation componentwise).

Recall from \cite[Chapter 1]{DN}  that for $p \in [1, \infty)$ and $k \geqslant 1$, the operator $D^k : \mathcal{S} \subset L^p(\Omega) \rightarrow L^p(\Omega; H^{\otimes k} )$ is closable. For $p \geqslant 1$ and $k \geqslant 1$ we define the seminorms
\begin{align}\label{def:norm-kp}
\| X \|_{k,p} = \left( \E \|X\|^p + \sum_{i=1}^k  \E \| D^i X \|_{H^{\otimes i} }^p \right)^{\frac{1}{p}}
\end{align}
We define $\mathbb{D}^{k,p} $ as the completion of the space $\mathcal{S}$ in $L^p(\Omega)$ with respect to $\| \cdot \|_{k,p}$. Furthermore, we denote for $k \geqslant 1$ the spaces
\begin{align*}
\mathbb{D}^{k, \infty} := \bigcap_{p \geqslant 1} \mathbb{D}^{k,p}, \qquad \mathbb{D}^\infty  := \bigcap_{k \geqslant 1} \mathbb{D}^{k, \infty},
\end{align*}
where the latter is a metric space that is complete with metric $d(f,g):=\sum_{k,p\geq 1}2^{-k-p}\frac{\|f-g\|_{\mathbb{D}^{k,p}}}{1+\|f-g\|_{\mathbb{D}^{k,p}} }
$ for any $f,g\in\mD^\infty$.

The adjoint of $D$ is denoted by $\delta$: the domain of $\delta$ is those elements $u\in L^2(\Omega, H)$ such that there exists $ Y\in L^2(\Omega)$ such that $\mE\langle u,DX\rangle_H=\mE( YX)$ for every $X\in \mathcal{S}$. We then write $Y=\delta u$.
If $u$ is an adapted (to the filtration generated by $W$) process such that $\int_{0}^{1}\mE|u_t|^2\dd t<\infty$, then $\delta u$ is its It\^o integral $\delta u=\int_{0}^{1}u_t\dd W_t$.
If furthermore $u_t\in\mD^{1,2}$ and $\int_{0}^{1}\mE|D_tu_s|^2\dd s<\infty$
then we have the following identity (see \cite[Proposition 3.8]{Hai})
\begin{align}\label{prop:commds}
  D_t ( \delta (u)) = u_t + \delta ( D_t u)=u_t+\int_{0}^{1} D_tu_s\dd W_s.
\end{align}

For a random vector $X = (X^1, \ldots, X^d)$ whose components are in $\mathbb{D}^{1,2}$, its \emph{Malliavin matrix} is defined as
 \begin{align*}
  \mathcal{M}^{ij}:=\langle D X^i,DX^j\rangle_H,\quad 1\leq i,j\leq d
 \end{align*}
 whenever it makes sense.
 We say that a random vector $X = (X^1, \ldots, X^d)$ whose components are in $\mathbb{D}^{1,2}$ is \emph{nondegenerate} if its Malliavin matrix $\mathcal{M}$ is a.s. invertible and $(\det \mathcal{M})^{-1} \in L^p(\Omega)$ for all $p \geqslant 1$.
 We also use the notation  $\mathcal{M}_t$ to denote the Malliavin matrix of $X_t$ for some $t\geq0$ when $(X_t)_{t\geq0}$ is a process. In this way we say a stochastic process $(X_t)_{t\geq0}$ is  \emph{nondegenerate} if for each $t\geq 0$, $X_t$ is a \emph{nondegenerate} random vector.

\begin{theorem}\cite[Proposition 2.1.4]{DN} \label{mc:thm1}
Let $X = (X^1, \ldots, X^d)$ be a nondegenerate random vector and fix $k \geqslant 1$. Suppose that $X^j \in \mathbb{D}^{k+1, \infty}$ for $j=1, \ldots, d$. For any $Y \in \mathbb{D}^{k, \infty}$ and any multi-index  $\alpha \in \mathbb{N}^d$ such that $|\alpha|\leq k$ define a random variable $Z_\alpha(X,Y)$ inductively by the recursion
\begin{align*}
Z_{e_j}(X,Y) &= \delta \Big( \sum_{l=1}^d (Y \cM^{-1})^{jl} DX^l \Big), \\
Z_{\alpha+e_j}(X,Y) &= Z_{e_j} (X, Z_\alpha(X,Y)).
\end{align*}
Then for every $\varphi \in C_p^\infty(\mathbb{R}^d)$ we have
\begin{align}\label{eq:Mal-deri}
\mE [(\partial_\alpha \varphi)(X)Y] = \mE[\varphi(X)Z_\alpha].
\end{align}
Furthermore, the random variables $Z_\alpha$ satisfy the bounds, for $1 \leqslant p < q < \infty$ with $\frac{1}{p} = \frac{1}{q} + \frac{1}{r}$,
\begin{align*}
\| Z_\alpha \|_{L^p_\omega} \leqslant N \, \| \cM^{-1} DX \|_{k,2^{k-1}r}^k \| Y \|_{k,q},
\end{align*}
where the constant $N$ depends on $p,q,d,k$.
\end{theorem}

The following result is an extension of \cite[Proposition 2.1.3]{DN} and \cite[Theorem 3.5, page 300]{KS}.

\begin{theorem} \label{lbl:thm1}
Let $t\in(0,1].$
Let $(u_s)_{s \in [0,1]}$ be a $\R^{d\times d_1}$-valued adapted process. Suppose that
\begin{itemize}
\item[(i)] $\mE \left( \int_0^1 \Vert u_s \Vert^2  \, \dd s \right) < \infty$,  $u_s \in \mathbb{D}^{1,2}$ for all $s \in [0,t]$, and for some $p \in [2, \infty)$ it holds that
\begin{align}
\mu := \sup_{s,r \in [0,1]} \mE \left( \Vert D_s u_r \Vert^p \right) < \infty, \label{lbl:thm1cond1}
\end{align}
\item[(ii)] there exists a constant $\lambda_* > 0$ such that $u_s u_s^* \geq \lambda_* I$ for all $s\in[0,t]$.
\end{itemize}
Set $X_t = \int_0^t u_s \, \dd W_s$ and denote by $\mathcal{M}_t$ the Malliavin matrix of $X_t$. Then for $\gamma \in (0, \frac{p}{2d})$ we have

\begin{align}\label{est:M-Malliavin-matrix}
\mE \left( (\det \mathcal{M}_t)^{-\gamma} \right) \leqslant N \, t^{- \gamma d} 
\end{align}
where the constant $N$ depends only on $\lambda_*, d, d_1, \mu,\gamma$ and $p$.
\end{theorem}

\begin{proof}

Note that for any $\xi\in\R^d$,
\begin{equ}
\xi^* \mathcal{M}_t \xi=\int_0^t\| D_s X_t\xi\|^2\dd s.
\end{equ}
We denote by $\lambda_t$ the smallest eigenvalue of $\mathcal{M}_t$. It is given by
\begin{align}
\lambda_t := \inf_{\| \xi \| = 1} \xi^* \mathcal{M}_t \xi = \inf_{\| \xi \| = 1} \int_0^t\| D_s X_t\xi\|^2\dd s.\label{lbl:deflambdt}
\end{align}
We notice that because of $\det \mathcal{M}_t \geqslant \lambda_t^d$ we have
\begin{align}
\mE ((\det \mathcal{M}_t)^{-\gamma}) \leqslant \mE (\lambda_t^{-d \gamma}). \label{lbl:bound5}
\end{align}
We therefore estimate this last quantity.
Note that \eqref{prop:commds} yields for $s < t$
\begin{align}\label{formu:inte}
D^k_s X^m_t = u^{k,m}_s + \sum_{l=1}^{d_1} \int_s^t D^k_s (u_r^{m,l}) \, \dd W_r^l.
\end{align}
Using the elementary inequality $(a+b)^2 \geqslant \frac{1}{2}a^2 - b^2$ we obtain, for any $\xi\in\R^d$ with $\|\xi\|=1$,
\begin{align}
\| D_s X_t\xi\|^2 &
\geqslant
\frac{1}{2} \| u_s\xi\|^2
-\Big\|\Big(\int_s^t D_su_r\dd W_r\Big)\xi\Big\|^2
\geqslant \frac{1}{2} \lambda_*
-\Big\|\int_s^t D_su_r\dd W_r\Big\|^2.\label{lbl:bound1}
\end{align}
The lower bound for the first term is provided by the ellipticity assumption in $(ii)$.
After plugging (\ref{lbl:bound1}) into (\ref{lbl:deflambdt}) we arrive at the following estimates for any $h \in [0,1]$
\begin{align}
\lambda_t &\geqslant \int_{t(1-h)}^t  \frac{1}{2}\lambda_* -\Big\|\int_s^t D_su_r\dd W_r\Big\|^2  \dd s
=: \frac{1}{2}th\lambda_* - I_h(t).\label{soonfinal}
\end{align}

Going back to the quantity we want to estimate, we write, for any $a>0$,
\begin{align}
\mE (\lambda_t^{-\gamma d}) &
= \int_0^\infty \mP (\lambda_t^{-\gamma d} > y) \dd y \lesssim \int_0^\infty y^{\gamma d -1} \mP(\lambda_t^{-1} > y) \, \dd y \nonumber \\
 &\lesssim a^{\gamma d} + \int_a^\infty  y^{\gamma d -1} \mP \big(\lambda_t < 1/y \big) \dd y. \label{lbl:bound2}
\end{align}
We now choose the parameters as $a:= \frac{4}{t \lambda_*}$ and for any $y\geq a$, $h= \frac{4}{t \lambda_* y}$. Note that a $y \geqslant a$ indeed implies $h \leqslant 1$. From \eqref{soonfinal} we see

\begin{align}
\mP \big(\lambda_t < 1/y \big) \leqslant \mP \big(I_h(t) \geq 1/(2y) \big) \lesssim y^{\frac{p}{2}}  \mE ( | I_h (t) |^{\frac{p}{2}} ). \label{lbl:bound3}
\end{align}
Furthermore, applying the Jensen and  Burkholder-Davis-Gundy (BDG) inequalities yields
\begin{align}
\mE ( | I_h (t) |^{\frac{p}{2}} ) &\lesssim (th)^{\frac{p}{2} -1}  \int_{t(1-h)}^t \mE \Big| \int_s^t D_s u_r \, \dd W_r \Big|^p \, \dd s \nonumber \\
 &\lesssim (th)^{\frac{p}{2} -1} \int_{t(1-h)}^t \mE \Big( \int_s^t \|D_s u_r \|^2 \, \dd r \Big)^{\frac{p}{2}} \, \dd s \nonumber \\
 &\lesssim (th)^{\frac{p}{2} -1} \int_{t(1-h)}^t (t-s)^{\frac{p}{2} -1} \int_s^t  \sup_{r \in [0,1]} \mE  \|D_s (u_r) \|^p {\dd r} \, \dd s \nonumber \\
 &\lesssim  (th)^p \sup_{r,s \in [0,1]} \mE \|D_s (u_r) \|^p  \; .
 \label{lbl:bound4}
\end{align}
Since $th\approx y^{-1}$, we incorporate (\ref{lbl:bound4}) and (\ref{lbl:bound3}) into (\ref{lbl:bound2}) and obtain for $\gamma \in (0, \frac{p}{2d})$
\begin{align}
\mE (\lambda_t^{-\gamma d}) &\lesssim t^{-\gamma d} +  \int_{\frac{4}{t \lambda_*}}^\infty \mE ( | I^k_h (t) |^{\frac{p}{2}} ) \, y^{\gamma d -1 + \frac{p}{2}} \dd y \nonumber\\
 &\lesssim   t^{- \gamma d} + \int_{\frac{4}{t \lambda_*}}^\infty y^{\gamma d -1 - \frac{p}{2}} \dd y 
\lesssim t^{-\gamma d}.\label{lbl:bound6}
\end{align}
Returning with the above to (\ref{lbl:bound5}) finishes the proof.
\end{proof}
\begin{remark}\label{rem:est-theta}
    We will later need a slight extension of \cref{lbl:thm1}, to accommodate processes of the form $X_t^\theta:=\int_0^{t_1}u_s\dd W_s+\theta\int_{t_1}^tu_s\dd W_s$, $0< t_1\leq t$,  $\theta\in[0,1]$. Bounding the Malliavin matrix of $X$ follows the same lines as above, for the sake of completeness we provide the argument.
    \begin{align*}
     \|  &  D_sX_t\xi\|^2
     \\=&  \| (   D_sX_t\mathds{1}_{s\in[0,t_1]}+D_sX_t\mathds{1}_{s\in(t_1,t]})\xi\|^2
\\   \geq&
\Big(\frac{1}{2} \| u_s\xi\|^2
-\Big\|\Big(\int_s^{t_1} D_su_r\dd W_r\Big)\xi\Big\|^2 \Big)\mathds{1}_{s\in(0,t_1)}+
\Big(\frac{\theta^2}{2} \| u_s\xi\|^2
-\Big\|\Big(\theta\int_s^t D_su_r\dd W_r\Big)\xi\Big\|^2\Big)\mathds{1}_{s\in[t_1,t]}
\\ \geq& \Big(\frac{1}{2} \lambda_*
-\Big\|\int_s^{t_1} D_su_r\dd W_r\Big\|^2\Big)\mathds{1}_{s\in(0,t_1)}+ \theta^2\Big(\frac{1}{2} \lambda_*
-\Big\|\int_s^t D_su_r\dd W_r\Big\|^2\Big)\mathds{1}_{s\in[t_1,t]}
    \end{align*}
which yields for any $h\in[0,1]$
\begin{align*}
 \lambda_t\geq& \int^{t_1}_{t_1(1-h)}  \Big(\frac{1}{2} \lambda_*
-\Big\|\int_s^{t_1} D_su_r\dd W_r\Big\|^2\Big)\mathds{1}_{s\in(0,t_1)}\dd r
\\&+  \int_{t_1}^{t_1+(t-t_1)h}\theta^2\Big(\frac{1}{2} \lambda_*
-\Big\|\int_s^t D_su_r\dd W_r\Big\|^2\Big)\mathds{1}_{s\in(t_1,t]}\dd r
\\\geq& \frac{1}{2} \lambda_* (t_1h+\theta^2(t-t_1)h)-\int^{t_1}_{t_1(1-h)}
\Big\|\int_s^{t_1} D_su_r\dd W_r\Big\|^2\mathds{1}_{s\in(0,t_1)}\dd r
\\&-  \int_{t_1}^{t_1+(t-t_1)h}\theta^2\Big\|\int_s^t D_su_r\dd W_r\Big\|^2\mathds{1}_{s\in(t_1,t]}\dd r \\=&:\frac{1}{2} \lambda_* (t_1h+\theta^2(t-t_1)h)-I_{ 1,{h}}(t_1)-I_{ 2,{h}}^\theta(t)
=:\frac{1}{2} \lambda_* (t_1h+\theta^2(t-t_1)h)-I_h^\theta(t).
\end{align*}
Next, we use \eqref{lbl:bound2}, and this time we take $a:=\frac{4}{t_1\lambda_*+\theta^2(t-t_1)\lambda_*}$ and for any $y\geq a$, $h=\frac{4}{y(t_1\lambda_*+\theta^2(t-t_1)\lambda_*)}\leq 1$.
Similarly to \eqref{lbl:bound3}, we have
\begin{align*}
    \mP \big(\lambda_t < 1/y \big) \leqslant \mP \big(I_h^\theta(t) \geq 1/(2y) \big) \lesssim y^{\frac{p}{2}}  \mE ( | I_h^\theta (t) |^{\frac{p}{2}} ).
\end{align*}
Similarly to \eqref{lbl:bound4} we have
\begin{align*}
  &\mE ( | I_{ 1,{h}} (t_1) |^{\frac{p}{2}} )\lesssim (t_1h)^p \sup_{r,s \in [0,1]} \mE \|D_s (u_r) \|^p\lesssim y^{-p} ,
  \\&\mE ( | I_{ 2,{h}} (t) |^{\frac{p}{2}} )\lesssim \theta^{p}((t-t_1)h)^p \sup_{r,s \in [0,1]} \mE \|D_s (u_r) \|^p\lesssim y^{-p}.
\end{align*}
Finally, similarly to \eqref{lbl:bound6}, we get for any $\gamma\in(0,\frac{p}{2d})$
\begin{align*}
    \mE (\lambda_t^{-\gamma d}) &\lesssim (t_1+\theta^2(t-t_1))^{-\gamma d}.
\end{align*}
It shows that instead of \eqref{est:M-Malliavin-matrix}, the Malliavin matrix $\mathcal{M}_t^\theta$ of $X_t^\theta$ satisfies the following bound:
\begin{align}\label{est:theta-M-Mal}
    \mE \left( (\det \mathcal{M}_t^\theta)^{-\gamma} \right) \leqslant N \, (t_1+\theta^2(t-t_1))^{- \gamma d}.
\end{align}
\end{remark}
\medskip

Define for $j\in\mN$, $p\geq1$, and random variables $X$, the quantity
\begin{align}
\cD(j,p)(X) := \esssup_{s_1, \ldots, s_j \in [0,1]} \mE
\| {D_{s_1, \ldots, s_j}^j} X\|^p, \label{pr:bound1}
\end{align}
whenever it is finite. Define the following class of processes
\begin{align*}
    \frD^k=\{(X_t)_{t\in[0,1]}: &\,\cD(j,p)(X_t)<\infty\,\, \forall 1 \leqslant j \leqslant k,p \geqslant 2, t\in[0,1]
    \\&\text{and }D_s X_t = 0\,\,\forall s > t\}.
\end{align*}

\begin{proposition} \label{pr:prop1}
Let $k\in\N$ and $(X_t)_{t\in[0,1]}\in\frD^{k+1}$. Then for all $p \geqslant 2$, $t\in(0,1]$, one has
\begin{align}
\mE \| D^{k+1} X_t \|_{H^{\otimes (k+1)}}^p &\leqslant \cD(k+1,p)(X_t) \, t^{\frac{p}{2} (k+1)},\label{est:bounde1}\\
\mE \|D^k \mathcal{M}_t \|^p_{H^{\otimes k}  } &\leqslant N \, t^{\frac{p}{2}(k+2)}, \label{pr:bound2}
\end{align}
where $N$ depends only on $k,d$, and finitely many $\cD(j,r)(X_t)$ with $j=1, \ldots, k+1$ and $r \geqslant 2$.
\end{proposition}
\begin{remark}
    For the above statement (and several below) $X_t$ does not actually need to be a process, the statement holds with any random variable $X$ satisfying the condition $D_s X=0$, $s>t$, for some fixed $t$.
    We choose to formulate the statements like this to stay closer to the standard literature, and since in their applications we will use them with processes.
\end{remark}
\begin{proof}
We first note that because of our assumption, we have $D{^{k+1}_{s_1, \ldots , s_{k+1}}} X_t = 0$ if $s_1 \vee \cdots \vee s_{k+1} > t$. Applying Jensen's inequality we then obtain
\begin{align*}
\mE \| D^{k+1} X_t \|_{H^{\otimes k+1} }^p &= \mE \Big( \int_0^t \cdots \int_0^t {\|} D{^{k+1}_{s_1, \ldots , s_{k+1}}}X_t {\|}^2 \dd s_1 \ldots \dd s_{k+1} \Big)^{\frac{p}{2}} \\
 &\leqslant t^{(\frac{p}{2}-1)(k+1)} \,  \int_0^t \cdots \int_0^t \mE\|  D{^{k+1}_{s_1, \ldots , s_{k+1}}} X_t \|^p \dd s_1 \ldots \dd s_{k+1},
\end{align*}
from which the first bound follows immediately.

For the second bound note that for any $m,q \in \lbrace 1, \ldots, d \rbrace$, by applying Jensen's inequality twice we get
\begin{align*}
\mE& \|D^k \mathcal{M}^{mq}_t \|^p_{H^{\otimes k} } \\&= \mE \Big( \int_0^t \cdots \int_0^t \left| D{^{k}_{s_1, \ldots, s_k}}\int_0^t \langle D_s X^m_t, D_s X^q_t \rangle_{\mR^{d_1}} \, \dd s \right|^2 \dd s_1 \ldots \dd s_k \Big)^{\frac{p}{2}} \\
 &\leqslant  \mE \Big( \int_0^t \cdots \int_0^t t \int_0^t \Big| D{^{k}_{s_1, \ldots, s_k}} \langle D_s X^m_t {, }  \, D_s X^q_t\rangle_{{\mR^{d_1}} }  \Big|^2 \dd s \,  \dd s_1 \ldots \dd s_k \Big)^{\frac{p}{2}} \\
 &\lesssim t^{\frac{p}{2}(k+2)} \sup_{s, s_1, \ldots, s_k\in(0,t]} \mE
   \big| D{^{k}_{s_1, \ldots, s_k}} \langle D_s X^m_t    {, } \,  D_s X^q_t\rangle_{{\mR^{d_1}} }  \big|^p.
\end{align*}
We can then use the following Leibniz rule:
\begin{align*}
D^k_{s_1, \ldots, s_k} (FG) = \sum_{I \subset \lbrace s_1, \ldots, s_k \rbrace} D_I^{|I|} (F) \; D_{I^c}^{k-|I|} (G), \quad F,G \in \mathbb{D}^{k,p}
\end{align*}
for any $k \geqslant 1$, $p \geqslant 2$. Applying this with $F=D^l_s X^m_t$ and $G=D^l_s X^q_t$ and subsequently using H\"older's inequality leads to \eqref{pr:bound2}.
\end{proof}

\begin{proposition}
Let $k\in\N$ and let $(X_t)_{t\in[0,1]}\in\frD^{k+1}$ be nondegenerate.
Then for all $p\geqslant 2$, $t\in(0,1]$, one has
\begin{align}
\mE \| D^k \mathcal{M}^{-1}_t \|_{H^{\otimes k}}^p &\leqslant N \, t^{\frac{p}{2}k} \sum_{j=1}^k  [\mE | \det \mathcal{M}_t |^{-2(1+2j)p}]^\frac{j+1}{2(1+2j)} t^{p(d(j+1)-1)} \label{pr:bound3}
\end{align}
where $N$ depends only on $k$, $d$, and finitely many $\cD(j,r)(X_t)$ with $j=1, \ldots, k+1$ and $r \geqslant 2$.
\end{proposition}

\begin{proof}
Because of \cite[Lemma 2.1.6]{DN} we have $$D\mathcal{M}^{-1}_t = -\mathcal{M}^{-1}_t (D\mathcal{M}_t) \mathcal{M}^{-1}_t.$$ For higher order derivatives it follows by induction that there are constants
$C(k,j,c_1, \ldots, c_j)$ such that
\begin{align*}
D^k \mathcal{M}^{-1}_t = \sum_{j=1}^k \sum_{\substack{1 \leqslant c_1, \ldots, c_j \leqslant k \\ c_1 + \cdots + c_j = k}} C(k,j,c_1, \ldots, c_j) \left[ \mathcal{M}^{-1}_t \prod_{i=1}^j \left( (D^{c_i} \mathcal{M}_t) \mathcal{M}^{-1}_t \right) \right].
\end{align*}
By the above, (\ref{pr:bound2}) and H\"older's inequality, we deduce
\begin{align}
\mE \| D^k \mathcal{M}^{-1}_t \|_{H^{\otimes k} }^p &\lesssim \sum_{j=1}^k [\mE\|  \mathcal{M}^{-1}_t  \|^{p(1+2j)}]^{\frac{j+1}{1+2j}}  \sum_{\substack{1 \leqslant c_1, \ldots, c_j \leqslant k \\ c_1 + \cdots + c_j = k}} \prod_{i=1}^j [\mE\|  D^{c_i} \mathcal{M}_t  \|_{H^{\otimes c_i} }^{p(1+2j)}]^{\frac{1}{1+2j}} \nonumber \\
 &\lesssim \sum_{j=1}^k [\mE\|  \mathcal{M}^{-1}_t  \|^{p(1+2j)}]^{\frac{j+1}{1+2j}}  \,  \sum_{\substack{1 \leqslant c_1, \ldots, c_j \leqslant k \\ c_1 + \cdots + c_j = k}} t^{\frac{p}{2} \sum_{i=1}^j (c_i+2)}  \nonumber \\
 &\lesssim \sum_{j=1}^k[\mE\|  \mathcal{M}^{-1}_t  \|^{p(1+2j)}]^{\frac{j+1}{1+2j}} \,  t^{\frac{p}{2} (k + 2j)} \label{pr:eq1}
\end{align}
Recall that there is a constant $C=C(d)$ such that for {any non-degenerate} $d\times d$ matrix $A$,  $\Vert A^{-1} \Vert  \leqslant C \, |\text{det}(A)|^{-1}\|A\|^{d-1}$. Therefore, we have
\begin{align*}
\mE\|  \mathcal{M}^{-1}_t  \|^{p(1+2j)} &\leqslant [\mE | \det \mathcal{M}_t |^{-2(1+2j)p}]^\frac{1}{2} \; [\mE\| \mathcal{M}_t \|^{2(1+2j)p(d-1)}]^{\frac{1}{2}} \\
 &\lesssim \, [\mE | \det \mathcal{M}_t |^{-2(1+2j)p}]^\frac{1}{2} \; t^{(1+2j)p(d-1)}.
\end{align*}
Plugging the above inequality into \eqref{pr:eq1} leads promptly to \eqref{pr:bound3}.
\end{proof}

\begin{theorem} \label{pr:thm1}
Let $k\in\N$ and let $(X_t)_{t\in[0,1]}\in\frD^{k+1}$ be nondegenerate. Furthermore, assume that
such that for all $p\geq 1$ there exists $C_p>0$  such that for all $t\in(0,1]$
\begin{align}\label{333}
\mE | \det \mathcal{M}_t |^{-p} \leqslant C_p t^{-p d}.
\end{align}
 Then for all $Y \in \mathbb{D}^{k, q}$ with $q > 1$, $\varphi \in C_p^\infty(\mR^d)$, and multiindex $\alpha$ with $|\alpha| = k$ one has
\begin{align}
| \mE \partial_\alpha \varphi(X_t) Y | \leqslant N \| \varphi \|_{C^0} \| Y \|_{k,q} \, t^{-\frac{k}{2} }, \label{pr:boundsg2}
\end{align}
for all $t \in (0,1]$, where $N$ depends only on $k$, $d$, and finitely many $C_r$ and $\cD(j,r)(X_t)$ with $j=1, \ldots, k+1$ and $r \geqslant 2$.
\end{theorem}

\begin{proof}
Using $Z_\alpha$ from  \cref{mc:thm1}, we can write
\begin{align}
| \mE \partial_\alpha \varphi(X_t) Y | \leqslant \| \varphi \|_{C^0} \, \| Z_\alpha \|_{L^1_\omega} \lesssim \| \varphi \|_{C^0} \, \| \mathcal{M}_t^{-1} DX_t \|_{k,p}^{k} \| Y \|_{k,q}, \label{pr:bound5}
\end{align}
where $p=2^{k-1}\frac{q-1}{q}$.
In the first step, we note that
\begin{align*}
D^k (\mathcal{M}^{-1}_t DX_t) = \sum_{j=0}^k \binom{k}{j} D^j \mathcal{M}^{-1}_t \otimes D^{k-j+1} X_t.
\end{align*}
Furthermore by \eqref{pr:bound3} and \eqref{333} for $t \in [0,1]$ we obtain { for any $j\leq k$}
\begin{align*}
\mE \| D^j \mathcal{M}_t^{-1} \|_{H^{\otimes j} }^p\lesssim t^{\frac{p}{2} j} \sum_{i=1}^j t^{-p} \lesssim  t^{\frac{p}{2}(j -2)}.
\end{align*}
Putting it together with \eqref{est:bounde1} leads to
\begin{align*}
\mE \| D^k (\mathcal{M}^{-1}_t DX_t) \|_{H^{\otimes k} }^p &\lesssim \sum_{j=0}^k \big\| \| D^j M^{-1}_t \|_{H^{\otimes j} }^p \big\|_{L^2_\omega} \; \big\| \| D^{k-j+1} X_t \|_{H^{\otimes k-j+1}}^p \big\|_{L^2_\omega} \\
 &\lesssim  \sum_{j=0}^k t^{\frac{p}{2}(j -2)} \, t^{\frac{p}{2}(k-j+1)}
 \lesssim  \sum_{j=0}^k t^{\frac{p}{2}(k - 1)} \lesssim t^{\frac{p}{2}(k- 1)}.
\end{align*}
Therefore
\begin{align}\label{est-inter-Malliavin}
\| \mathcal{M}^{-1}_t DX_t \|_{k,p} \lesssim  \sum_{j=0}^{k} t^{\frac{1}{2}(j-1)}\lesssim  t^{-\frac{1}{2}},
\end{align}
and therefore from \eqref{pr:bound5} we get \eqref{pr:boundsg2}.
\end{proof}

\subsection{Estimates on the laws of certain processes}\label{sec:driftless-est}
Throughout this subsection we fix $1\geq\lambda>0$ and $K>0$
and assume that $\sigma\sigma^*$ satisfies \eqref{def-uni-ellip} with $\lambda$, and that $\|\sigma\|_{\mathcal{C}^1}\leq K$.
For now we consider the Milstein scheme without drift term:
\begin{align}\label{eq:Milstein-scheme-driftless}
  d\bar X_t^n=\Big(\sigma(\bar X_{k_n(t)}^n)+\nabla\sigma\sigma(\bar X_{k_n(t)}^n)(W_t-W_{k_n(t)})\Big)\dd W_t,\quad \bar X_0^n=y.
\end{align}
Estimates on the law of $\bar X^n$ will later be transferred to $X^n$ by a Girsanov transform.
Fix a function $\chi\in \mathcal{C}_0^\infty(\mathbb{R};\mathbb{R})$ such that $|\chi(x)|\leq |x|$ and
\begin{align}\label{def:chi}
\chi(x)=\left\{
\begin{array}{cc}
  x & \text{ if } \quad |x|\leq \frac{\kappa}{2}, \\
  0 & \text{ if } \quad |x|\geq \kappa.
\end{array}\right. \quad\quad\quad\kappa:=\frac{\lambda}{4Kd^2},
\end{align}
Note that for any $k\in\mathbb{N}$, $\vert \nabla^k \chi \vert\leq N(\lambda,K,d,k)$.
Introduce the truncated Milstein scheme corresponding to \eqref{eq:Milstein-scheme-driftless}: for any $y \in \mathbb{R}^d$ define the process $(\hat X^{n}_t (y))_{t\in[0,1]} $ by
\begin{align}\label{eq:hatX}
\dd \hat X^{n,i}_{t}  = \sum_{j=1}^{d_1}\Big( \sigma^{ij} (\hat  X^{n}_{k_n(t)} ) + \sum_{k=1}^{d_1}(\nabla \sigma \sigma)^{ijk}(\hat  X^{n}_{k_n(t)} ) \chi (W_t^k - W_{\kappa_n (t)}^k) \Big) \dd W_t^j, \quad \hat  X^{n}_0 = y.
\end{align}
Define the event $\hat\Omega\subset \Omega$ by
\begin{align}\label{def:hat-Omega}
\hat{\Omega} := \left\{ \sup_{t \in \left[\frac{k-1}{n}, \frac{k}{n} \right]} \vert W^l_t - W^l_{\frac{k}{n}} \vert \leqslant \kappa /2, \; \forall k=1,\ldots,n,\quad l=1,\ldots,d_1\right\}.
\end{align}
Note that by the assumptions on $\sigma$, $(\hat X_t^n)_{t\in[0,1]}$ coincides with $(\bar X_t^n)_{t\in[0,1]}$ on $\hat{\Omega}$.
Analogously to \cite[Proposition 5.3]{BDG} we know that there exist constants $N$ and $c > 0$ which depend only on $d$ and $\kappa$, such that
\begin{align}\label{def:pro-event}
 \mP (\hat{\Omega}) \geqslant 1 - N e^{-cn}.
\end{align}
Further, we define the auxiliary processes, $\bar Z_t(x)=(\bar Z_t^i(x))_{i=1,\ldots,{d}}$, for $x\in\R^d$, by
\begin{align}\label{def:rv-Z}
\bar Z_t^i(x):=\sum_{j=1}^{d_1}\int_{0}^{t}\sigma^{ij}(x)\dd W_r^{j}+\sum_{j,l=1}^{d_1}\int_{0}^{t}[\nabla\sigma\sigma]^{ijl} (x)\chi( W_r^j)\dd W_r^l.
\end{align}

\begin{lemma}
  \label{lemma-density-continuous}
  Let $q > 1$, $k\in\mN$, $Y \in \mathbb{D}^{k,q}$,  and $G\in \mathcal{C}^k(\mathbb{R}^{d})$. Then there exists a constant $N$ depending on $\kappa, d, d_1, \lambda, K,k$,   such that for any multi-index $\alpha$ with $|\alpha|=k$ and $t\in(0,1]$ one has the bound
  \begin{align}\label{expect-Mallianvin-1}
   \sup_{x\in\mR^d}| \mE[\partial^\alpha G(\bar Z_t(x)) Y]|\leq N \Vert G\Vert_{\infty} \| Y\|_{k,q} \, t^{-\frac{k}{2}}.
  \end{align}
\end{lemma}
\begin{proof}
We apply \cref{lbl:thm1} and \cref{pr:thm1}. Fix $x\in\mR^d$ and for simplicity we drop it from the notation of $\bar Z(x)$. Denote $A:=\sigma(x), B:=(\nabla\sigma\sigma)(x)$, $\chi(W_t):=(\chi(W_t^j))_{1\leq j\leq d_1}$. It is evident that $|A|,|B|,|\chi|,|\chi'|\lesssim 1$.
Let $u_t=A+B\chi(W_t)$. Firstly we have
\begin{align*}
  \mE\big[\int_{0}^{t}|u(t)|^2\dd t\big]<\infty,\quad \mu:=\sup_{s,t\in[0,1]}\mE[|D_su_t|^p]= \sup_{s\in[0,t]}\mE[|B\chi'(W_t)|^p]<\infty,  \quad \forall  p\geq 2.
\end{align*}
Therefore $u_t\in\mathbb{D}^{1,2}$ for all $t\in[0,1]$. Moreover
\begin{align*}
  u_tu_t^*= AA^{*}+AB^*\chi(W_t)+B\chi(W_t)A^*+B\chi(W_t)B^*\chi(W_t)\geq \frac{1}{4}\lambda I.
\end{align*}
Then \cref{lbl:thm1} implies that for the Malliavin matrix $\mathcal{M}_t$ of $\bar Z_t$, we have for any $\gamma\in(0,\infty)$, $t\in(0,1]$
\begin{align*}
\mE[(\text{det}\mathcal{M}_t)^{-\gamma}]\lesssim_\gamma t^{-\gamma d}.
\end{align*}
Moreover, as in \eqref{formu:inte}, we have  for $s_1,s_2,s_3,\ldots, s_{k+1}<t$,
 \begin{align*}
D_{s_1} \bar Z_t =& u_{s_1} +  \int_{s_1}^t B\chi'(W_r)  \dd W_r,\\
{D_{s_2s_1}^2} \bar Z_t =& B\chi'(W_{s_1\vee s_2})+\int_{s_1\vee s_2}^t B\chi''(W_r)  \dd W_r,\quad\cdots\\
{D_{s_{k+1}\ldots s_3s_2s_1}^{k+1}} \bar Z_t =& B\partial^k\chi(W_{s_1\vee s_2\vee s_3\vee\ldots\vee s_{k+1}})+\int_{s_1\vee s_2\vee s_3\ldots \vee s_{k+1}}^t B\partial^{k+1}\chi(W_r)  \dd W_r.
\end{align*}
From here it follows that $\bar Z\in\frD^{k+1}$
and therefore \cref{pr:thm1} yields \eqref{expect-Mallianvin-1}. 
\end{proof}

\begin{lemma} \label{lemma-density-descrete}
Let   $q\geq 2$, $m\in\N$.  Let $Y \in \mathbb{D}^{m,q}$, $G\in \mathcal{C}^m(\mathbb{R}^{d})$, and assume $\sigma\in{\mathcal{C}^{m+2}}$.
  Let $(\hat X^{n}_t)_{t\in[0,1]}$ be the solution to \eqref{eq:hatX}.  Then there exists a constant $N$ 
      depending on $ \kappa, d, d_1,\lambda,K,k$, and  $\Vert\sigma\Vert_{\mathcal{C}^{m+2}}$
      such that for any multi-index $\alpha$ with $|\alpha|=m$ all $t\in(0,1]$ one has the bound
  \begin{align}
   \sup_{x\in\R^d}\big| \mE[\partial^\alpha G(\hat X^{n}_t(x)) Y]\big|&\leq N \Vert G\Vert_{L^\infty_x} \| Y \|_{m,q} \, t^{-\frac{m}{2}}.\label{expect-Mallianvin-hat0-infty}
  \end{align}
\end{lemma}
\begin{proof}
We want to conclude  \eqref{expect-Mallianvin-hat0-infty} by applying \cref{lbl:thm1} and \cref{pr:thm1}.
First we want to show that $\sigma\in{\mathcal{C}^{m+2}}$ implies $(\hat X^n_t)_{t\in[0,1]}\in\frD^{m+1}$, with furthermore
\begin{align}\label{bound:finite-derivative}
\cD(1,r)(\hat X_t^n),\ldots,\cD(m+1,r)(\hat X_t^n)\lesssim_r 1.
\end{align}
Since $D_s\hat X^n_t=0$ for $s>t$ is obvious, one only needs to show \eqref{bound:finite-derivative}.
We only detail the argument for showing $\cD(1,r)(\hat X_t^n)\lesssim_r 1$ which corresponds to $m=0$ (for $m> 0$ the bound can be obtained by the { same} induction argument as in e.g. \cite[Proof of Proposition 5.2]{Hai}).
The argument will actually give more: for any $s$ that is not a gridpoint, we show
\begin{equ}\label{eq:almostdone1}
    \E\sup_{t\in[s,1]}\|D_s \hat X^n_t\|^r\lesssim_r 1.
\end{equ}
Set $k_n^+(s)=k_n(s)+1/n$. First note that for $t\in[s,k_n^+(s)]$ we have
\begin{align}
D_s{}\hat  X_t^{n}=\sigma(\hat  X_{k_n(s)}^n)&+\nabla\sigma\sigma(\hat  X_{k_n(s)}^n)\chi(W_s-W_{k_n(s)})
+
\int_{s}^{t}\nabla\sigma\sigma(\hat  X_{k_n(s)}^n)\chi'(W_r-W_{k_n(s)})\dd W_r.\nonumber
\end{align}
From $\sigma\in\mathcal{C}^2$ (here even $\sigma\in\mathcal{C}^1$ suffices) {together with BDG’s inequality} it is clear that
\begin{equ}\label{eq:almostdone}
    \E\sup_{t\in[s,k_n^+(s)]}\|D_s \hat X^n_t\|^r\lesssim_r 1.
\end{equ}
If $t>k_n^+(s)$, one can inductively obtain
\begin{align}
D_s{}\hat  X_t^{n}=&D_s{}\hat  X_{k_n(t)}^{n}
+\int_{k_n(t)}^{t}\big(\nabla \sigma(\hat  X_{k_n(t)}^n)+\nabla(\nabla\sigma\sigma)(\hat  X_{k_n(t)}^n)\chi(W_r-W_{k_n(t)})\big)^*D_s{}\hat  X_{k_n(t)}^{n}\dd W_r\nonumber
\\
=&D_s{}\hat  X_{k_n^+(s)}^{n}+
\int_{k_n^+(s)}^{t}\big(\nabla \sigma(\hat  X_{k_n(r)}^n)+\nabla(\nabla\sigma\sigma)(\hat  X_{k_n(r)}^n)\chi(W_r-W_{k_n(r)})\big)^*D_s{}\hat  X_{k_n(r)}^{n}\dd W_r.\nonumber
\end{align}
This is simply a linear delay equation for $t\mapsto D_s \hat X^n_t$ with bounded coefficients, since  $\sigma\in\mathcal{C}^2$.
From here it is classical (by BDG and Gronwall inequalities) to get
\begin{align}\label{est:Ds-hat-Xn-1}
     \mE\sup_{t\in[k_n^+(s),1]}\Vert D_s\hat  X_{t}^{n}\Vert^{r}\lesssim_r \mE\Vert D_s\hat  X_{k_n^+(s)}^{n}\Vert^{r},
\end{align}
which, combined with \eqref{eq:almostdone} yields the claimed bound \eqref{eq:almostdone1}.

  Now we  denote
  \begin{align*}
A_t:=\sigma({}\hat  X_{k_n(t)}),\quad \quad B_t:=(\nabla\sigma\sigma)({}\hat  X_{k_n(t)}^n)\chi(W_t-W_{k_n(t)})
  \end{align*}
and set  $u_t :=  A_t+B_t$. By the definition of $\chi$, $A_t$ and $B_t$ one then has
  \begin{align}\label{con:elliptic-diffusion}
   u_t u_t^* =(A_t+B_t)(A_t+B_t)^*\geq \frac{\lambda}{4}I.
  \end{align}
Condition \eqref{lbl:thm1cond1} holds for any $p\geq 2$ since $\sigma\in\mathcal{C}^2$ and one can apply \cref{lbl:thm1} to $u_t$. Thus for the Malliavin matrix $\mathcal{M}_t$ of $\int_0^t u_r dW_r = \hat X^{n}_t$ holds
\begin{align*}
\mE [(\det \mathcal{M}_t)^{-\gamma} ] \lesssim_\gamma \, t^{- \gamma d}
\end{align*}
for $t \in (0,1]$ and all $\gamma\geq 2$. Hence $\hat X_t^n$ is non-degenerate and $\hat X_t^n\in\mD^{1,p}$ for all $p>2$ and \eqref{bound:finite-derivative} holds for $m=0$ as claimed.
By \cref{pr:thm1} we  get \eqref{est-inter-Malliavin}.
\end{proof}

\section{Intermediate estimates on Milstein scheme}\label{sec:inte-est-Mil}
Before showing the desired estimates we recall the following version of { the} \emph{stochastic sewing lemma}, originating from \cite{Le}.
Let \begin{align*}
&[S,T]_\leq^2:=\left\{(s,t)\in[S,T]^2:S\leq s\leq t\leq T\right\},
  \\  &[S,T]_\leq^3:=\left\{(s,u,t)\in[S,T]^3:S\leq s\leq u\leq t\leq T\right\}.
    \end{align*}
    Given a two-parameter process $(s,t)\mapsto A_{s,t}$, we set for $(s,u,t)\in[S,T]_\leq^3$
    $$ \delta A_{s,u,t}:=A_{s,t}-A_{u,t}-A_{s,u}.$$
    \begin{lemma}
        \cite[Theorem {2.3}]{Le}\label{lemm:sewing-lemm} Let $p\geq2$, $0\leq S\leq T\leq 1$. Let {$(A_{s,t})_{S\leq s\leq t\leq T }$} be a two parameter field with values in $\R^d$ such that for each $s\leq t$,  $A_{s,t}$ is $\cF_t$-measurable. Suppose that for some $\epsilon_1,\epsilon_2>0$ and $C_1,C_2$ the bounds
        \begin{align}
         \Vert A_{s,t}\Vert_{L_\omega^{p}}&\leq C_1|t-s|^{\frac{1}{2}+\epsilon},\label{sewing-con-1}\\
         \Vert \mE_s\delta A_{s,u,t}\Vert_{L_\omega^{p}}&\leq C_2|t-s|^{1+\epsilon_2},\label{sewing-con-2}
         \end{align}
         hold for all $(s,u,t)\in[S,T]_\leq^3$. Then there exists a unique (up to modification) $(\mathcal{F}_t)_{t\in[0,1]}$-adapted process $\mathcal{A}:[S,T]\rightarrow L_p(\Omega)$ such that $\mathcal{A}_S=0$ and the following bounds hold for some constants $K_1$, $K_2>0$:
         \begin{align}
         \Vert \mathcal{A}_t-\mathcal{A}_s-A_{s,t}\Vert_{L_\omega^{p}}&{\leq K_1|t-s|^{\frac{1}{2}+\epsilon_1}},\quad (s,t)\in[S,T]_\leq^2\label{lem:sewing-est-1}\\
         \Vert \mE_s[\mathcal{A}_t-\mathcal{A}_s-A_{s,t}]\Vert_{L_\omega^{p}}&\leq K_2|t-s|^{1+\epsilon_2},\quad (s,t)\in[S,T]_\leq^2.\label{lem:sewing-est-2}
         \end{align}
          Moreover, there exists a constant $K$ depending only on $\epsilon_1$ and $\epsilon_2,d$ such that $\mathcal{A}$ satisfies the bound
    \begin{align}
        \label{lem:sewing-est-3}
        \Vert \mathcal{A}_t-\mathcal{A}_s\Vert_{L_\omega^{ p}}&\leq KpC_1|t-s|^{\frac{1}{2}+\epsilon_1}+KpC_2|t-s|^{1+\epsilon_2},\quad (s,t)\in[S,T]_\leq^2.
    \end{align}
    \end{lemma}

 \subsection{Estimates on additive functionals}
\begin{lemma}\label{lem:driftless}
 Let $\alpha\in(0,1)$,  $\epsilon\in(0,\frac{1}{2})$, $\alpha'\in(1-2\epsilon,1), p\geq 2$.   Suppose  that $(H^\sigma)$ in \cref{ass:main} holds and that $ (\hat X_t^n(y))_{t\in[0,1]}$ is the solution to \eqref{eq:hatX}, $y\in\mathbb{R}^d$. Then for all  functions $h\in \mathcal{C}^{\alpha'}, f\in \mathcal{C}^\alpha$, $(s,t)\in[0,1]_\leq^2$, the following holds
 \begin{align}\label{est:driftless}
  \Big \Vert \int_s^t h(\hat X^n_r)\big(f(\hat X^n_r)-f(\hat X^n_{k_n(r)})\big)\dd r\Big\Vert_{L^p_\omega}\leq N\Vert f\Vert_{\mathcal{C^\alpha}}\Vert h\Vert_{\mathcal{C}^{\alpha'}}n^{-\frac{\alpha+1}{2}+\epsilon}|t-s|^{\frac{1}{2}+\epsilon}
 \end{align}
 with $N=N(p,d, d_1, \Vert\sigma\Vert_{\mathcal{C}^3},\alpha,\alpha',\lambda, \epsilon)$.
\end{lemma}
\begin{proof}
  We partially follow and partially refine the arguments of \cite[Lemma 3.1]{DGL} (see also \cite[Lemma 6.1]{BDG}). { By the standard approximation argument it is sufficient to show that \eqref{est:driftless} holds for  smooth $f,h$. }
   Define $k$ by $\frac{k}{n}=k_n(s)$ for $s\in[0,1]$.  Let
  \begin{align*}
A_{s,t}:=\mE_s[\mathcal{A}_{s,t}]:=\mE_s\int_{s}^{t}h(\hat X^n_s)(f(\hat X_r^n)-f(\hat X_{k_n(r)}^n))\dd r.
  \end{align*}
  In order to apply \cref{lemm:sewing-lemm}, we are going to  verify \eqref{sewing-con-1} first. Using the $\cF_s$-measurability of $h(\hat X^n_s)$, we write
  \begin{align*}
    \Vert  A_{s,t}\Vert_{L^p_\omega}\leq \Vert h\Vert_{\mB}\tilde A_{s,t}:=  \Vert h\Vert_{\mB}\int_{s}^{t}\big\Vert\mE_s\big[f(\hat X_r^n)-f(\hat X_{k_n(r)}^n)\big]\big\Vert_{L^p_\omega}\dd r.
  \end{align*}
Depending on the relation of the various variables,
there are several trivial cases, which we deal with first.
If $t\in[s,\frac{k+4}{n}]$, then using the bound
\begin{align}\label{est:diff-barX}
 \big \Vert \sup_{r\in[0,1]}|\hat X_r^n-\hat X_{k_n(r)}^n  \big\Vert_{L^m_\omega}\lesssim_m n^{-\frac{1}{2}},
\end{align}
for any $m\in(0,\infty)$,
we get for any $\epsilon\in(0,\frac{1}{2})$
 \begin{align}\label{est:case-I}
 \tilde A_{s,t}&
   \lesssim \Vert f\Vert_{\mathcal{C}^\alpha}\int_{s}^{t}n^{-\frac{\alpha}{2}}\dd r\lesssim \Vert f\Vert_{\mathcal{C}^\alpha}n^{-\frac{1+\alpha}{2}+\epsilon}|t-s|^{\frac{1}{2}+\epsilon}
 \end{align}
 since $|t-s|\leq 4n^{-1}$.
In the sequel we assume $t>\frac{k+4}{n}$. We write
  \begin{align}\label{eq:A}
  \tilde A_{s,t}
   & \leq \Big(\int_{s}^{\frac{k+4}{n}}+\int^{t}_{\frac{k+4}{n}}\Big)\big\Vert\mE_s\big[f(\hat X_r^n)-f(\hat X_{k_n(r)}^n)\big]\big\Vert_{L^p_\omega}\dd r
   =:S_1+S_2.
  \end{align}
  The term $S_1$ is as simple as before:
  \begin{align}\label{est:S1-Ho}
    S_1\lesssim \Vert f\Vert_{\mathcal{C}^{\alpha}}n^{-\frac{1+\alpha}{2}+\epsilon}|t-s|^{\frac{1}{2}+\epsilon}.
  \end{align}
  For $S_2$, notice that
  \begin{align*}
    S_2=\int^{t}_{\frac{k+4}{n}}\big\Vert \mE_s\mE_{\frac{k+1}{n}}(\mE_{k_n(r)}f(\hat X_r^n)-f(\hat X_{k_n(r)}^n))\big\Vert_{L^p_\omega}\dd r.
  \end{align*}
 For $x,y\in\mR^d$, let
  \begin{align*}
    \Sigma(x,y):&=\nabla\sigma\sigma(x)\chi(y),\quad \quad   A(x):=(\sigma\sigma^*)(x),
    \\B(x,y):&=(\Sigma\Sigma^*)(x,y),\quad \quad C(x,y):=\sigma(x)\Sigma(x,y),
  \end{align*}
  and define 
  \begin{align}\label{eq:g}
    g(x):=g_r^n(x):=\mE\Big(f\big(x+\int_{k_n(r)}^{r}[\sigma(x)+\Sigma(x,W_\xi-W_{k_n({\color{blue}\xi})})]\dd W_\xi\big)-f(x)\Big).
  \end{align}
  Then by the Markov property
  \begin{align}\label{eq:rewrite-g}
     S_{2}=\int^{t}_{\frac{k+4}{n}}\big\Vert \mE_s\mE_{\frac{k+1}{n}}g(\hat X_{k_n(r)}^n)\big\Vert_{L^p_\omega} dr.
  \end{align}
  For simplicity let
  \begin{align*}
    Y_{u,v}(x):&={\int_{u}^{v}[\sigma(x)+\Sigma(x,W_\xi-W_{k_n(\xi)})]\dd W_\xi}.
  \end{align*}
  When $u=k_n(v)$ one can rewrite this quantity as
  \begin{equ}
          Y_{k_n(v),v}(x)={\int_{k_n(v)}^{v}[\sigma(x)+\Sigma(x,W_\xi-W_{k_n(v)})]\dd W_\xi}.
  \end{equ}
  By It\^o's formula we have
  \begin{align}\label{eq:re-g}
    g(x)=
\sum_{i,j=1}^{d}\mE\int_{k_n(r)}^{r}\big[\big(&\frac{1}{2}A^{ij}(x)+\frac{1}{2}B^{ij}
    (x,W_u-W_{k_n(r)})
    \nonumber\\&+C^{ij}(x,W_u-W_{k_n(r)})\big)\partial_{ij}f(x+Y_{k_n(r),u}(x))\big]\dd u.
  \end{align}
We aim to show for $\alpha\in(0,1)$
  \begin{align}\label{est:regularity-alpha}
  \big|\mE_{\frac{k+1}{n}}g(\hat X_{k_n(r)}^n(y))\big|
  \lesssim \Vert f\Vert_{\mathcal{C}^\alpha}n^{-\frac{1+\alpha}{2}}(r-s)^{-\frac{1}{2}}.
  \end{align}
 We claim that it suffices to show \eqref{est:regularity-alpha} in the cases $\alpha=0,1$.
Indeed, the general case of \eqref{est:regularity-alpha} then follows via interpolation (i.e.  applying \cite[Theorem 1.6]{AL} and \cite[1.1.1. Example 1.8]{AL}).

First we treat the case $\alpha=0$.
 {
 Since $\hat X^n$ is a homogeneous Markov process on the grid $0,1/n,2/n,\ldots$,  and the random field $Y_{k_n(r),r}(\cdot)$ is independent from $\hat X^n_{k_n(r)}(y)$ and $Y_{k_n(r),r}\overset{\mathrm{law}}{=}Y_{k_n(r)-\frac{k+1}{n},r-\frac{k+1}{n}}$ by its definition,
 we can apply the Markov property;  together with the Newton-Leibniz formula we write, with denoting $k_n(r)^-:=k_n(r)-\frac{k+1}{n}, r^-:=r-\frac{k+1}{n}$}
 \begin{align}\label{eq:Markov}
     &\mE_{\frac{k+1}{n}}g(\hat X_{k_n(r)}^n(y))
    \nonumber\\&{= \mE\big(g_{r^-}^n(\hat X_{k_n(r)-\frac{k+1}{n}}^n(z))\big)\big|_{z=\hat X_{\frac{k+1}{n}}^n(y)}}\nonumber\\
     &{=\mE\Big[\int_0^1\nabla f\big(\hat X_{k_n(r)^-}^n(z)+\theta Y_{k_n(r)^-,r^-}(\hat X_{k_n(r)^-}^n(z))\big)\cdot Y_{k_n(r)^-,r^-}(\hat X_{k_n(r)^-}^n(z))\dd \theta\Big]\bigg|_{z=\hat X_{\frac{k+1}{n}}^n(y)}.}
 \end{align}

 \noindent
We want to get rid of $\nabla$ by Malliavin integration by parts.
More precisely, setting {$\hat X_{r^-}^{n,\theta}(z):=\hat X_{k_n(r)^-}^n(z)+\theta Y_{k_n(r)^-,r^-}(\hat X_{k_n(r)^-}^n(z)),$ 
 we wish to use \eqref{pr:boundsg2} 
with $k=1$, $q=2$, $\hat X^{n,\theta}(z)$ in place of $X$, and $Y_{k_n(r)^-,r^-}(\hat X_{k_n(r)^-}^n(z))$
in place of $Y$.  {In the following without confusion we shortly denote $\tilde{Y}
:=Y_{k_n(r)^-,r^-}(\hat X_{k_n(r)^-}^n(z))$}.
It is easy to verify that $ \hat X^{n,\theta}(z)\in \frD^{2}$: indeed, $D_s \hat X^{n,\theta}_{r^-}(z)=0$ for $s>r^-$ is obvious, while to bound $\cD(j,p)(\hat X^{n,\theta}_{r^-}(z))$, we can proceed as follows.}
For $s\in({k_n(r)^-,r^-}]$ we have
\begin{align}
&D_s{}\hat  X^{n,\theta}_{r^-}({z})\nonumber\\&=D_s{}\hat  X^{n}_{k_n(r)^-}({z})+\theta D_s{\tilde{Y}}
\nonumber\\&=\theta\big(\nabla\sigma\sigma(\hat  X^n_{k_n(r)^-}({z}))\chi(W_s-W_{k_n(r)^-})
+
\int_{s}^{r^-}\nabla\sigma\sigma(\hat  X^n_{k_n(r)^-}({z}))\chi'(W_\xi-W_{k_n(r)^-})\dd W_\xi\big),\nonumber
\end{align}
and for $s<{k_n(r)^-}$,
\begin{align}
D_s{}\hat  X^{n,\theta}_{r^-}({z})
=D_s\hat  X^n_{k_n(r)^-}+\theta\int^{r^-}_{k_n(r)^-}&\nabla\big(\sigma(\hat  X^n_{k_n(r)^-}({z}))\nonumber\\&+\nabla\sigma\sigma(  \hat  X^n_{k_n(r)^-}({z}))\chi(W_\xi-W_{k_n(r)^-})
\big)D_s\hat    X^n_{k_n(r)^-}({z})\dd W_\xi.\nonumber
\end{align}
From $\hat X^n\in\frD^2$, BDG  inequality, and $\sigma\in
\mathcal{C}^2$ we get $\cD(1,p)(\hat X^{n,\theta}_{r^-}({z}))\lesssim_p 1$. Bounding $\cD(2,p)(\hat X^{n,\theta}_{r^-}({z}))$ is very similar. This verifies $ \hat X^{n,\theta}({z})\in \frD^{2}$.
To verify \eqref{333}, we apply \cref{rem:est-theta}. Letting $u_\xi:=\sigma(\hat X^n_{k_n(\xi)})+\Sigma(\hat X^n_{k_n(\xi)},W_\xi-W_{k_n(\xi)})$, the conditions of \cref{rem:est-theta} are satisfied with $t_1=k_n(r)^-$. Therefore
the Malliavin matrix $\mathcal{M}^\theta_{r^-}$ of $\hat X^{n,\theta}_{r^-}({z})$ satisfies
\begin{align*}
\mE | \det \mathcal{M}_{r-}^\theta |^{-p} \leqslant C_p \big(k_n(r)-\frac{k+1}{n}+\theta^2(r-k_n(r))\big)^{-p d}\lesssim C_p(r-s)^{-pd}.
\end{align*}
Therefore
 $\hat X^{n,\theta}({z})$ fulfills \eqref{333}.
Finally, it remains to bound the $\|\cdot\|_{1,2}$ norm of ${\color{blue}\tilde{Y}}$.
We have
\begin{align*}
   D_s {\tilde{Y}}
    &=\mathds{1}_{s\in [k_n(r)^-,r^-]}\Big(\sigma(\hat X^n_{k_n(r)^-}({z}))+\Sigma(\hat X^n_{k_n(r)^-}({z}),W_\xi-W_{k_n(r)^-})\\&\qquad\qquad\qquad\qquad\qquad+\int_{s}^{r^-}\nabla\sigma\sigma(\hat X^n_{k_n(r)^-}({z}))\chi'(W_\xi-W_{k_n(r)^-})\dd W_\xi\Big)\\
    &\,\,+\mathds{1}_{s< k_n(r)^-}\int^{r^-}_{k_n(r)^-}\nabla [\sigma(\hat X^n_{k_n(r)^-}({z}))+\Sigma(\hat X^n_{k_n(r)^-}({z}),W_\xi-W_{k_n(r)^-})] D_s\hat X^n_{k_n(r)^-}({z})\dd W_\xi
\end{align*}
and
\begin{align*}
   \Vert {\tilde{Y}}
   \Vert_{1,2} &\lesssim
   \Vert {\tilde{Y}}
   \Vert_{L_\omega^2}+ \big\Vert \big(\int_0^1| D_s{\tilde{Y}}|^2\dd s\big)^{\frac{1}{2}}\big\Vert_{L^2_\omega}
   \\&\lesssim  \Vert{\tilde{Y}}\Vert_{L^2_\omega}+\big(\int_0^1\mathds{1}_{s\in [k_n(r)^-,r^-]}(s)\dd s\big)^{\frac{1}{2}}
 \\&\,\,\,+\big\Vert \big(\int^{r^-}_{k_n(r)^-}\big|\int_s^{r^-} \nabla\sigma\sigma(\hat X^n_{k_n(r)^-}({z}))\chi'(W_\xi-W_{k_n(r)^-})\dd W_\xi \big|^2\dd s\big)^{\frac{1}{2}}\big\Vert_{L^2_\omega}
 \\&\,\,\, +\Vert \hat X^n_{k_n(r)^-}({z})\Vert_{1,4}  \Big\Vert \int^{r^-}_{k_n(r)^-}\nabla [\sigma(\hat X_{k_n(r)}^n(y))+\Sigma(\hat X^n_{k_n(r)^-}({z}),W_\xi-W_{k_n(r)^-})] \dd W_\xi\Big\Vert_{L^4_\omega}
   \\ &\lesssim n^{-\frac{1}{2}}.
\end{align*}
Putting the above together, we apply \eqref{pr:boundsg2} to get { for any $z\in\mR^d$}
 \begin{align*}
 &{\mE\Big(\int_0^1\nabla f\big(\hat X_{k_n(r)^-}^n(z)+\theta Y_{k_n(r)^-,r^-}(\hat X_{k_n(r)^-}^n(z))\big)\cdot Y_{k_n(r)^-,r^-}(\hat X_{k_n(r)^-}^n(z))\dd \theta\Big)}
  \\&=\E\Big(\int_0^1 \nabla f\big(\hat X^n_{k_n(r)^-}(z)+\theta {\tilde{Y}}\big)\cdot {\tilde{Y}}\dd \theta \Big)
   \\& \lesssim (r-s)^{-\frac{1}{2}}\Vert f\Vert_{\mathcal{C}^0}\Vert {\tilde{Y}}\Vert_{1,2}
   \\&\lesssim (r-s)^{-\frac{1}{2}}\Vert f\Vert_{\mathcal{C}^0} n^{-\frac{1}{2}}.
 \end{align*}
So {after plugging the above into \eqref{eq:Markov}, we get }\eqref{est:regularity-alpha} holds with $\alpha=0$.\\

Next, we consider $\alpha=1$. By \eqref{eq:re-g} and   {  the Markov property similarly to \eqref{eq:Markov}}
\begin{align}\label{est:middle-g}
 &  \mE_{\frac{k+1}{n}}g(\hat X_{k_n(r)}^n(y)) \nonumber\\=&\sum_{i,j=1}^{d}\int_{k_n(r)}^{r}  \mE_{\frac{k+1}{n}}\big[\big(\frac{1}{2}A^{ij}(\hat X_{k_n(r)}^n(y))+\frac{1}{2}B^{ij}
    (\hat X_{k_n(r)}^n(y),W_u-W_{k_n(r)})
    \nonumber\\&\qquad\qquad+C^{ij}(\hat X_{k_n(r)}^n(y),W_u-W_{k_n(r)})\big)\partial_{ij}f\big(\hat X_{u}^n(y)\big)\big]\dd u
    \nonumber\\=&:\sum_{i,j=1}^{d}\int_{k_n(r)}^{r} \mE_{\frac{k+1}{n}}\big[\Gamma^{ij}(\hat X_{k_n(r)}^n(y),W_u-W_{k_n(r)})\partial_{ij}f\big(\hat X_{u}^n(y)\big)\big]\dd u
     \nonumber\\=&:{\sum_{i,j=1}^{d}\int_{k_n(r)}^{r}\mE\big[\Gamma^{ij}(\hat X_{k_n(r)^-}^n(z),W_{u^-}-W_{k_n(r)^-})\partial_{ij}f\big(\hat X_{u^-}^n(z)\big)\big]\big|_{z=\hat X_{\frac{k+1}{n}}^n(y)}}\dd u.
\end{align}
Again, we apply \eqref{pr:boundsg2}
 with $k=1$, $q=2$, $\hat X^n_{u^-}({z})$ in place of $X$, and $\Gamma^{ij}(\hat X^n_{k_n(r)^-}({z}),W_{\color{blue}u^-}-W_{k_n(r)^-})$ in place of $Y$.
The conditions on $X$ follow immediately from \eqref{bound:finite-derivative} and the nondegeneracy of the  stochastic integrand. It is also easy to check that
\begin{align*}
    \Vert & \Gamma^{ij}(\hat X^n_{k_n(r)^-}({z}),W_{u^-}-W_{k_n(r)^-})\Vert_{1,2}\lesssim 1.
\end{align*}
Therefore by \eqref{pr:boundsg2} we get {for any $z\in\mR^d$}
 \begin{align*}
 &{\mE\big[\Gamma^{ij}(\hat X_{k_n(r)^-}^n(z),W_{u^-}-W_{k_n(r)^-})\partial_{ij}f\big(\hat X_{u^-}^n(z)\big)\big]}
    \\&\lesssim \Vert\Gamma^{ij}(\hat X^n_{k_n(r)^-}({z}),W_{u^-}-W_{k_n(r)^-})\Vert_{1,2} \Vert f\Vert_{\mathcal{C}^1}(u-\frac{k+1}{n})^{-\frac{1}{2}} \lesssim \Vert f\Vert_{\mathcal{C}^1}(r-s)^{-\frac{1}{2}}.
 \end{align*}
Using it in \eqref{est:middle-g} and integrating, we get
 \begin{align*}
    \mE_{\frac{k+1}{n}}g(\hat X_{k_n(r)}^n(y))\lesssim  \Vert f\Vert_{\mathcal{C}^1}(r-s)^{-\frac{1}{2}} n^{-1}.
 \end{align*}
This finishes the proof of \eqref{est:regularity-alpha} for $\alpha=1$, and therefore for all $\alpha\in[0,1]$.
Now we use \eqref{est:regularity-alpha} with the $\alpha$ of the theorem, in \eqref{eq:rewrite-g}. Using also the fact that $n^{-1}\lesssim|t-s|$, we get
  \begin{align}\label{est:S21-Ho}
  S_{2}\lesssim&\Vert f\Vert_{\mathcal{C}^{\alpha}}n^{-\frac{1+\alpha}{2}}\int^{t}_{\frac{k+4}{n}} (r-s)^{-\frac{1}{2}}dr
  \lesssim\Vert f\Vert_{\mathcal{C}^{\alpha}}n^{-\frac{1+\alpha}{2}}|t-s|^{\frac{1}{2}}\lesssim \Vert f\Vert_{\mathcal{C}^{\alpha}}n^{-\frac{1+\alpha}{2}+\epsilon}|t-s|^{\frac{1}{2}+\epsilon}.
  \end{align}
Combining  \eqref{eq:A}, \eqref{est:S1-Ho} 
and \eqref{est:S21-Ho}, we get
  \begin{align}\label{est:case-II}
   \Vert   A_{s,t}\Vert_{L_\omega^p}\leq \Vert h\Vert_{\mB}\tilde A_{s,t}\leq N\Vert f\Vert_{\mathcal{C}^{\alpha}}\Vert h\Vert_{\mB}n^{-\frac{1+\alpha}{2}+\epsilon}|t-s|^{\frac{1}{2}+\epsilon}.
  \end{align}
We conclude that that \eqref{sewing-con-1} holds with
  $C_1=N\Vert h\Vert_{\mathcal{C}^{\alpha'}}\Vert f\Vert_{\mathcal{C}^{\alpha}}n^{-\frac{1+\alpha}{2}+\epsilon}$.

  Now we move on to verifying \eqref{sewing-con-2}. We have
  \begin{align*}
     \delta A_{s,u,t}= \Big(\mE_s\int_{u}^{t} h(\hat X_s) -\mE_u\int_{u}^{t}h(\hat X_u)\Big)\big[f(\hat X_r^n)-f(\hat
    X_{k_n(r)}^n)\big]\dd r,
  \end{align*}
and thus
  \begin{align*}
      \mE_s\delta A_{s,u,t}=\int_{u}^{t}\mE_s \Big(\big(h(\hat X_s) -h(\hat X_u)\big)\mE_u\big[f(\hat X_r^n)-f(\hat
    X_{k_n(r)}^n)\big]\Big)\dd r.
  \end{align*}
Using the preceding discussion, particularly the notation from \eqref{est:case-I} and the estimate \eqref{est:case-II}, we can write
  \begin{align*}
     \big\Vert\mE_s\delta A_{s,u,t}\big\Vert_{L^p_\omega}  &\lesssim\int_{u}^{t}\big\Vert  \big(h(\hat X_s) -h(\hat X_u)\big\Vert_{L_\omega^{2p}}\big\Vert\mE_u\big[f(\hat X_r^n)-f(\hat
    X_{k_n(r)}^n)\big]\big\Vert_{L^{2p}_\omega}\dd r
     \\&\lesssim |u-s|^{\frac{\alpha'}{2}}\Vert h\Vert_{\mathcal{C}^{\alpha'}} \int_{u}^{t}\big\Vert\mE_u\big[f(\hat X_r^n)-f(\hat
    X_{k_n(r)}^n)\big]\big\Vert_{L^{2p}_\omega}\dd r
     \\&\lesssim |u-s|^{\frac{\alpha'}{2}}\Vert h\Vert_{\mathcal{C}^{\alpha'}} \tilde A_{u,t}
      \\&\lesssim |u-s|^{\frac{\alpha'}{2}}\Vert h\Vert_{\mathcal{C}^{\alpha'}} \Vert f\Vert_{\mathcal{C}^\alpha} n^{-\frac{1+\alpha}{2}+\epsilon}|t-u|^{\frac{1}{2}+\epsilon}
       \\&\lesssim |t-s|^{1+\epsilon'}\Vert h\Vert_{\mathcal{C}^{\alpha'}} \Vert f\Vert_{\mathcal{C}^\alpha} n^{-\frac{1+\alpha}{2}+\epsilon}
  \end{align*}
  with $\epsilon':=\frac{\alpha'-1}{2}+\epsilon$, which is positive by assumption. Therefore \eqref{sewing-con-2} holds as well with $C_2=N\Vert h\Vert_{\mathcal{C}^{\alpha'}} \Vert f\Vert_{\mathcal{C}^\alpha} n^{-\frac{1+\alpha}{2}+\epsilon}$.

So all of the conditions in \cref{lemm:sewing-lemm}  are satisfied. Notice that
the process
\begin{align*}
    \hat{\mathcal{A}}_t:=\int_0^t h(\hat X^n_r)\big(f(\hat X^n_r)-f(\hat X^n_{k_n(r)})\big)\dd r
\end{align*}
is $\mathcal{F}_t$-adapted; moreover it clearly satisfies the following bounds
\begin{align*}
&    \big \Vert \hat{\mathcal{A}}_t-\hat{\mathcal{A}}_s-A_{s,t}\big\Vert_{L_\omega^{p}}=\big \Vert \hat{\mathcal{A}}_{s,t}-A_{s,t}\big\Vert_{L_\omega^{p}}\lesssim |t-s|,\\
   & \big \Vert \mE_s[\hat{\mathcal{A}}_t-\hat{\mathcal{A}}_s-A_{s,t}]\big\Vert_{L_\omega^{p}}=\big \Vert \mE_s[\hat{\mathcal{A}}_{s,t}-A_{s,t}]\big\Vert_{L_\omega^{p}}\lesssim |t-s|^{1+\frac{\alpha'}{2}}.
\end{align*}
Hence \cref{lemm:sewing-lemm} shows that $\hat{\mathcal{A}}=\mathcal{A}$ and the desired estimate \eqref{est:driftless} follows from \eqref{lem:sewing-est-3}.
\end{proof}
After applying Kolmogorov continuity theorem to \cref{lem:driftless} we simply get the following corollary.
\begin{corollary}\label{lemm:driftless-hf-Holder}
   Let $\alpha\in(0,1)$,  $\epsilon\in(0,\frac{1}{2})$, $\alpha'\in(1-2\epsilon,1), p\geq 2$.   Suppose  that $(H^\sigma)$ in \cref{ass:main} holds and that $ (\hat X_t^n(y))_{t\in[0,1]}$ is the solution to \eqref{eq:hatX}, $y\in\mathbb{R}^d$. Then for all  functions $h\in \mathcal{C}^{\alpha'}, f\in \mathcal{C}^\alpha$,  { for all }  $(s,t)\in[0,1]_\leq^2$, the following holds
 \begin{align}\label{est:driftless-hf}
   \Big\Vert \sup_{t\in[0,1]}\big|\int_0^t h( \hat X^n_r)\big(f( \hat X^n_r)-f(\hat X^n_{k_n(r)})\big)\dd r\big|\Big\Vert_{L^p_\omega}\leq N\Vert h\Vert_{\mathcal{C^{\alpha'}}}\Vert f\Vert_{\mathcal{C^\alpha}}n^{-\frac{\alpha+1}{2}+\epsilon}
 \end{align}
 with $N=N(p,d,d_1,\Vert\sigma\Vert_{\mathcal{C}^3},\alpha,\alpha',\lambda,\epsilon)$.
\end{corollary}

\subsection{Girsanov transform}
We add back the drift via a  Girsanov transform, first still in the truncated diffusion case. Therefore we use yet another auxiliary processes $(\tilde X_t^n(y))_{t\in[0,1]}=(\tilde X_t^n)_{t\in[0,1]}$, for $y\in\R^d$, defined by the following recursion
  \begin{align}\label{eq:tildeX}
    \tilde X_t^n=\tilde X_{k_n(t)}^n+\int_{k_n(t)}^{t}\sigma(\tilde X_{k_n(r)}^n)+\nabla\sigma\sigma(\tilde X_{k_n(t)}^n)\chi(W_r-W_{k_n(r)})\dd W_r+\int_{k_n(t)}^{t}b(\tilde X_{k_n(r)}^n)\dd r,
  \end{align}
 $\tilde X_{0}^n(y)=y$. Here $\chi:\mathbb{R}\mapsto\mR$ is defined as in \eqref{def:chi}, we again use the convention $\chi(x)=\chi(x_i)_{1\leq i\leq d_1}$ for $x\in\mR^{d_1}$, so $\chi(W_r-W_{k_n(r)}):=(\chi(W_r^i-W_{k_n(r)}^i))_{1\leq i\leq d_1}$.
\begin{corollary}
  \label{lem:drift-conditional-est}
  Let $\alpha\in(0,1)$,  $\epsilon\in(0,\frac{1}{2})$, $\alpha'\in(1-2\epsilon,1), p\geq 2$.   Suppose   \cref{ass:main} holds and that $ ( \tilde X_t^n(y))_{t\in[0,1]}$ is the solution to \eqref{eq:tildeX}, $y\in\mathbb{R}^d$. Then for all  functions $h\in \mathcal{C}^{\alpha'}, f\in \mathcal{C}^\alpha$, $(s,t)\in[0,1]_\leq^2$, the following holds
 \begin{align*}
   \Big\Vert \sup_{t\in[0,1]}\big|\int_0^t h( \tilde X^n_r)\big(f( \tilde X^n_r)-f(\tilde X^n_{k_n(r)})\big)\dd r\big|\Big\Vert_{L^p_\omega}\leq N\Vert h\Vert_{\mathcal{C^{\alpha'}}}\Vert f\Vert_{\mathcal{C^\alpha}}n^{-\frac{\alpha+1}{2}+\epsilon}
 \end{align*}
 with $N=N(p,d,d_1,\Vert\sigma\Vert_{\mathcal{C}^{3}},\alpha,\alpha',\Vert b\Vert_{\mathcal{C}^\alpha},\lambda,\epsilon)$.
\end{corollary}

\begin{proof}
  	For any  continuous process $Z$, we define
	\begin{align}\label{eq:cH}
	   \chh(Z)=\sup_{t\in[0,1]}\Big|\int_0^th(r,Z_r) [f(r,Z_r)-f(r,Z_{k_n(r)})]\dd r\Big|.
	\end{align}
 From \cref{lemm:driftless-hf-Holder} we have $\|\chh(\hat X^n)\|_{L^m_\omega}\lesssim_m \Vert h\Vert_{\mathcal{C^{\alpha'}}}\Vert f\Vert_{\mathcal{C^\alpha}}n^{-\frac{\alpha+1}{2}+\epsilon}$ for any $m<\infty$.
	Let $B_{k_n(r),r}( \cdot):=\nabla\sigma\sigma( \cdot)\chi(W_r-W_{k_n(r)})$. Define
	\begin{align*}
	    \rho:=\exp\Big(-\int_0^1((\sigma+&B_{k_n(r),r})^{-1}b)(\hat X_{k_n(r)}^n)\dd W_r
     \\&-\frac{1}{2}\int_0^1\Big|((\sigma+B_{k_n(r),r})^{-1}b)(\hat X_{k_n(r)}^n)\Big|^2\dd r\Big).
	\end{align*}
	By construction of $\chi$,  $(\sigma+B_{k_n(r),r})^{-1}b$ is bounded, and therefore { $\rho$ is an integrable random variable with expectation
equal to $1$.  } 
   Moreover, $\|\rho\|_{L^m_\omega}\lesssim_m 1$ for any $m<\infty$.
	 It follows from Girsanov theorem that under the measure $\rho d\mathbb{P}$, $\hat X^n$ is distributed the same as $\tilde X^n$ under $\mathbb{P}$, therefore by H\"older's inequality we can write
	\begin{align*}
	    \mE \chh(\tilde X^n)^{p}=\mE (\rho\chh(\hat X^n)^{ p})\leq \big[ \mE\chh(\hat  X^n)^{2 p}\big]^{1/2}\big[\mE\rho^{2}\big] ^{1/2}\lesssim \big(\Vert h\Vert_{\mathcal{C^{\alpha'}}}\Vert f\Vert_{\mathcal{C^\alpha}}n^{-\frac{\alpha+1}{2}+\epsilon}\big)^p.
	\end{align*}
\end{proof}

\begin{corollary}
  \label{coro-estimate-X}
  Let $\alpha\in(0,1)$,  $\epsilon\in(0,\frac{1}{2})$, $\alpha'\in(1-2\epsilon,1), p\geq 2$.   Suppose   \cref{ass:main} holds and that $ ( X_t^n(y))_{t\in[0,1]}$ is the solution to \eqref{eq:Milstein-scheme-SDE}, $x_0^n\in\mathbb{R}^d$. Then for all  functions $h\in \mathcal{C}^{\alpha'}, f\in \mathcal{C}^\alpha$, $(s,t)\in[0,1]_\leq^2$, the following holds
 \begin{align}\label{est:drift-hf-conditional}
   \Big\Vert \sup_{t\in[0,1]}\big|\int_0^t h(  X^n_r)\big(f( X^n_r)-f( X^n_{k_n(r)})\big)\dd r\big|\Big\Vert_{L^p_\omega}\leq N\Vert h\Vert_{\mathcal{C^{\alpha'}}}\Vert f\Vert_{\mathcal{C^\alpha}}n^{-\frac{\alpha+1}{2}+\epsilon}
 \end{align}
 with $N=N(p,d,d_1,\Vert\sigma\Vert_{\mathcal{C}^{3}},\alpha,\alpha',\Vert b\Vert_{\mathcal{C}^\alpha},\lambda,\epsilon)$.
\end{corollary}
\begin{proof}
   Recall that $( X_t^n)_{t\in[0,1]}$ coincides with $(\tilde X_t^n)_{t\in[0,1]}$ on $\hat \Omega$   and $\mP(\hat\Omega^c)\lesssim e^{-cn}$. Using the notation from \eqref{eq:cH} and applying \cref{lem:drift-conditional-est}, we can write
\begin{equs}
    \|\chh(X^n)\|_{L^p_\omega}&\leq\|\chh(\tilde X^n)\|_{L^p_\omega}+\big\|\mathds{1}_{\hat{\Omega}^c}\big(\chh(X^n)-\chh(\tilde X^n)\big)\big\|
    \\
    &\lesssim \Vert h\Vert_{\mathcal{C^{\alpha'}}}\Vert f\Vert_{\mathcal{C^\alpha}}n^{-\frac{\alpha+1}{2}+\epsilon}+\Vert f\Vert_{\mB}\Vert h\Vert_{\mB}\mP(\hat\Omega^c).
\end{equs}
This implies \eqref{est:drift-hf-conditional}.
\end{proof}

\section{Proof of the main result}\label{sec:main-proof}
The final ingredient of the proof is (an appropriate form) of the Zvonkin transformation, for which we need some regularity result for the PDE associated to \eqref{eq:SDE}.

\begin{assumption}
  \label{ass-a}
Let $a(x):\mR^d\mapsto\mR^{d\times d}$. Assume
\begin{itemize}
  \item[i).]there exists a constant $\lambda>0$ such that \eqref{def-uni-ellip} holds {with $A = a$},
  \item[ii).]$(a_{ij})_{1\leq i,j\leq d}$ is uniformly continuous with modulus of continuity $h$, in the sense that
      \begin{align}\label{def:uni-con}
      |a(x)-a(y)|\leq h(|x-y|) \quad \forall (x,y)\in  \R^{2d}.
      \end{align}
\end{itemize}
\end{assumption} For some parameter $\theta> 0$ we consider the elliptic equation
\begin{align}
    \label{PDE}
   \sum_{i,j=1}^da^{ij}\partial_{ij}u+\nabla u\cdot b-\theta u=f.
\end{align}
We then have the following results on solutions of \eqref{PDE}.
While such statements are fairly standard, in these particular forms we have not found an exact reference, so for the sake of completeness short proofs are provided in \cref{App:PDE}.

\begin{theorem}\label{thm:pde-Holder}
 Let $f,b\in\mathcal{C}^{\alpha}$ with {$\alpha\in(0,1)$}.  Suppose \cref{ass-a} holds with taking $h(x)=x^\alpha$ for $x\in\mR_+$.  There exists $\theta^*>0$ depending on $\alpha,\lambda$, $d$ and $\|b\|_{\mathcal{C}^{\alpha}}$ such that for $\theta\geq \theta^*$,  there exists a unique solution $u\in\mathcal{C}^{2+\alpha}$ to \eqref{PDE}. { Moreover,} for any $\gamma\in[\alpha,\alpha+2)$, there exists a constant $C$ depending on $\alpha$, $\lambda$, $d$, $\|b\|_{\mathcal{C}^{\alpha}}$ and $\gamma$, independent of $\theta$, such that
\begin{align}\label{est:pde-Holder}
  \Vert u\Vert_{\mathcal{C}^{\gamma}}\leq C\theta^{\frac{\gamma-(2+\alpha)}{2}}\|f\|_{\mathcal{C}^{\alpha}}.
\end{align}
\end{theorem}

In addition, let us recall two elementary properties of the approximation $X^n$  defined by \eqref{eq:Milstein-scheme-SDE}. For proofs, see \cite[Section 7.8.8]{Pages}.
\begin{proposition}\label{prop:kindof-trivial}
Assume that $b\in \mathcal{C}^0$ and $\sigma\in \mathcal{C}^1$. Then for any $m\in[1,\infty)$ there exists a constant $N=N(m, d,d_1,\|b\|_{\mathcal{C}^0},\|\sigma\|_{\mathcal{C}^1})$ such that for all $n\in\mathbb{N}$, $0\leq s\leq t\leq 1$, $f\in\mathcal{C}^2$ one has
\begin{equ}\label{eq:trivi-1}
\big\|X^n_t-X^n_s\big\|_{L^m_\omega}\leq N(t-s)^{1/2}
\end{equ}
and
\begin{equ}\label{eq:trivi-2}
\big\|f(X^n_t)-f(X^n_{\kappa_n(t)})-[\nabla f\sigma](X^n_{\kappa_n(t)})(W_t-W_{\kappa_n(t)})\big\|_{L^m_\omega}\leq N \|f\|_{\mathcal{C}^2}n^{-1}.
\end{equ}
\end{proposition}
\medskip

Now we are in position of giving the proof of the main theorem.
\begin{proof}[Proof of \cref{thm:main}]
Take $u$ to be the solution to \eqref{PDE} with $f=-b$, $a=\sigma\sigma^*$, and $\theta$ large enough which is determined later.
By \cref{thm:pde-Holder} we know $u\in{\mathcal{C}^{2+\alpha}}$.
  Then by It\^o's formula we have
\begin{align}
    \label{rep-b-X}
    \int_0^t b(X_r)dr=&u(x_0)-u(X_t)+\int_0^t[\nabla u\cdot \sigma](X_r)dW_r+\theta \int_0^tu(X_r)dr.
\end{align}
Similarly (using the summation convention for repeated indices $i,j,k$)
\begin{align}
    \label{rep-b-Xn}
     \int_0^t b(X_r^n)\dd r=&u(x_0^n)-u(X_t^n)+\theta \int_0^tu(X_r^n)\dd r
     \nonumber\\&+\int_0^t\nabla u(X_r^n)\cdot [\sigma(X_{k_n(r)}^n)+\nabla\sigma\sigma(X_{k_n(r)}^n)(W_r-W_{k_n(r)})]\dd W_r \nonumber\\&+
     \int_0^t\Big(\nabla u(X_r^n)\cdot[b(X_{k_n(r)}^n)-b(X_r^n)] \Big)\dd r
     \nonumber\\&+ \frac{1}{2}\int_0^t\Big(\partial_{ij}u(X_r^n)[a^{ij}(X_{k_n(r)}^n)-a^{ij}(X_{r}^n)]
    +B_r^{ij}(X_{k_n(r)}^n)\partial_{ij}u(X_r^n)\Big)\dd r
    \nonumber\\&+ \int_0^t\Big(\partial_{ij}u(X_r^n)[\sigma^{ki}(X_{k_n(r)}^n)\Sigma_r^{kj}(X_{k_n(r)}^n)]\Big)\dd r,
\end{align}
where $(B_r^{ij}(x))_{1\leq i\leq d,1\leq j\leq d}:=\Sigma_r(x)\Sigma_r(x)^*$ and  $\Sigma_r(x):=\nabla\sigma\sigma(x)(W_r-W_{k_n(r)})$.
 From equations \eqref{eq:SDE} and  \eqref{eq:Milstein-scheme-SDE} we have
 \begin{align*}
    X_t-X_t^n=x_0-x_0^n&+\int_0^t[b(X_r)-b(X_r^n)]\dd r+\int_0^t[b(X_r^n)-b(X_{k_n(r)}^n)]\dd r
    \nonumber\\&+\int_0^t[\sigma(X_r)-\sigma(X_{r}^n)]\dd W_r
     \nonumber\\&+\int_0^t[\sigma(X_{r}^n)-\sigma(X_{k_n(r)}^n)-(\nabla\sigma\sigma)(X_{k_n(r)}^n)(W_r-W_{k_n(r)})]\dd W_r.
 \end{align*}
 We use \eqref{rep-b-X} and \eqref{rep-b-Xn} to rewrite the first integral on the right-hand side. We raise to {$p$-th power for $p\geq 2$} and get
 \begin{align}\label{def:phi-X-Xn}
   \phi_t:=\sup_{s\in[0,t]} |X_s-X_s^n|^p\leq N\big(|x_0-x_0^n|^p+\sum_{\ell=1}^{4}V_t^\ell+\sum_{\ell=1}^{2}I_t^\ell\big),
 \end{align}
with
\begin{align*}
  V_t^1:&=|u(x_0)-u(x_0^n)|^p+\sup_{s\in[0,t]}\big[|u(X_s^n)-u(X_s)|^p+\theta^p\int_{0}^{s}|u(X_r)-u(X_r^n)|\dd r\big]^p,
  \\ V_t^2:&=\sup_{s\in[0,t]}\big|\int_0^s\Big((\nabla u(X_r^n)+I)(b(X_r^n)- b(X_{k_n(r)}^n))\Big)\dd r\big|^p,
   \\ V_t^3:&=\sup_{s\in[0,t]}\big|\int_0^s\partial_{ij}u(X_r^n)[a^{ij}(X_{r}^n)-a^{ij}(X_{k_n(r)}^n)]\dd r
   \big|^p,
     \\ V_t^4:&=\sup_{s\in[0,t]}\big|\int_0^s\Big(B_r^{ij}(X_{k_n(r)}^n)\partial_{ij}u(X_r^n)
     \Big)\dd r\big|^p,
      \\ V_t^5:&=\sup_{s\in[0,t]}\big|\int_0^s\Big(\partial_{ij}u(X_r^n)[\sigma^{ki}(X_{k_n(r)}^n)\Sigma_r^{kj}(X_{k_n(r)}^n)]\Big)\dd r\big|^p,
   \\ I_t^1:&=\sup_{s\in[0,t]}\big|\int_0^s\Big([(\nabla u+I)\cdot \sigma](X_r)-[(\nabla u+I)\cdot\sigma](X_r^n)\Big)\dd W_r\big|^p,
   \\ I_t^2:&=\sup_{s\in[0,t]}\big|\int_0^s\Big([\nabla u(X_r^n)+I]\cdot\big[\sigma(X_r^n)-\sigma(X_{k_n(r)}^n)-\nabla \sigma\sigma(X_{k_n(r)}^n)(W_r-W_{k_n(r)})\big]\Big)\dd W_r\big|^p.
\end{align*}

Now by \eqref{est:pde-Holder} we can take $\theta$ to be large enough so that $N\|\nabla u\|_{\mathbb{B}}\leq 1/4$, then we have the obvious bound
\begin{align}
\label{est:V1}
NV_t^1\leq \frac{1}{2}\sup_{s\in[0,t]}|X_s-X_s^n|^p+\theta^p\Vert u\Vert_{\mathcal{C}^1}\int_{0}^{t}\sup_{s\in[0,r]}|X_s-X_s^n|^p\dd r=\frac{1}{2}\phi_t+\theta^p\int_{0}^{t}\phi_r\dd r.
\end{align}
Applying \cref{coro-estimate-X} with $h=\nabla u+I$ and $f=b$, we get  
 \begin{align}\label{est:V2}
   \Vert V_t^2\Vert_{L^1_\omega}\lesssim  \big(n^{-\frac{1+\alpha}{2}+\epsilon}\big)^p.
 \end{align}
Similarly, but with the roles played by $h=\partial_{ij} u$, $f=a^{ij}$, we get
 \begin{align}\label{est:V3}
   \Vert V_t^3\Vert_{L_\omega^1}\lesssim&  \big(n^{-\frac{1+\alpha}{2}+\epsilon}\big)^p.
 \end{align}
  Since $\mE |B^{ij}(X_{k_n(r)}^n) |^p\lesssim \mE|W_r-W_{k_n(r)}|^{2p}\lesssim n^{-p}$, we also have
 \begin{align}\label{est:V4}
  \Vert V_t^4\Vert_{L^1_\omega}\lesssim \int_0^t\mE|B^{ij}(X_{k_n(r)}^n)|^p\dd r\lesssim Nn^{-p}.
 \end{align}
We manipulate the term $V^5$ as
 \begin{align*}
 \int_0^t&\Big(\partial_{ij}u(X_r^n)[\sigma^{ki}(X_{k_n(r)}^n)\Sigma^{kj}(X_{k_n(r)}^n)]\Big)\dd r\\
 =& \int_0^t\Big(\big[\partial_{ij}u(X_r^n)\sigma^{ki}(X_{r}^n)\big(\sigma^{ki}(X_{r}^n)
 -\sigma^{ki}(X_{k_n(r)}^n)\big)\big]
 \\&\quad\quad-\big[\partial_{ij}u(X_r^n)\big(\sigma^{ki}(X_{r}^n)-\sigma^{ki}(X_{k_n(r)}^n)\big)^2\big]
  \\&\quad\quad-\big[\partial_{ij}u(X_r^n)\sigma^{ki}(X_{k_n(r)}^n)\big(\sigma^{ki}(X_{r}^n)
  -\sigma^{ki}(X_{k_n(r)}^n)-\Sigma^{kj}(X_{k_n(r)}^n)\big)\big]\Big)\dd r
  \\=&:v^{51}_t-v^{52}_t-v^{53}_t.
 \end{align*}
Applying \cref{coro-estimate-X} with $h=\partial_{ij}u\sigma^{ki}$ and $f=\sigma^{ki}$, we get
 \begin{align*}
   \big \Vert| \sup_{s\in[0,t]}|v^{51}_s|^p\big\Vert_{L_\omega^1}\lesssim&  
   \big(n^{-\frac{1+\alpha}{2}+\epsilon}\big)^p.
 \end{align*}
From \eqref{eq:trivi-1} it is immediate that
 \begin{align*}
    \big\Vert \sup_{s\in[0,t]}|v^{52}_s|^p\big\Vert_{L_\omega^1}
    \lesssim n^{-p}.
 \end{align*}
Finally, from \eqref{eq:trivi-2} with $f=\sigma^{ik}$
 \begin{align*}
   \big\Vert \sup_{s\in[0,t]}|v^{53}_s|^p\big\Vert_{L_\omega^1}\lesssim&  
   n^{-p}.
 \end{align*}
Therefore
 \begin{align}\label{est:V5}
   \Vert V_t^5\Vert_{L_\omega^1}\lesssim&
   \big(n^{-\frac{1+\alpha}{2}+\epsilon}\big)^p.
 \end{align}
From  the pathwise  BDG inequality \cite[Theorem {3}]{Pietro}
it follows that there exist martingales $M^1$ and $M^2$ such that with probability  one
 \begin{align}\label{est:I1}
 |I_t^1|&\leq N \Big(\int_{0}^{t}\big|\nabla[(\nabla u+I)\cdot \sigma]\big|^2\big|X_r-X_r^n\big|^2\dd r\Big)^{\frac{p}{2}}+M_t^1
 \leq N \int_{0}^{t}\phi_r\dd r+M_t^1,
 \end{align}
 and
  \begin{align}\label{est:I2}
|I_t^2| \leq N  &\Big(\int_0^t\Big|[\nabla u(X_r^n)+I]\cdot\big[\sigma(X_r^n)-\sigma(X_{k_n(r)}^n)-\nabla \sigma\sigma(X_{k_n(r)}^n)(W_r-W_{k_n(r)})\big]\Big|^2\dd r\Big)^{\frac{p}{2}}
\nonumber\\&+M^2_t
\nonumber\\ \leq N & \int_0^t\Big|\sigma(X_r^n)-\sigma(X_{k_n(r)}^n)-\nabla \sigma\sigma(X_{k_n(r)}^n)(W_r-W_{k_n(r)})\Big|^p\dd r+M_t^2\nonumber\\=:\quad &I_t^{2,1}+M_t^2.
 \end{align}
Once again from \eqref{eq:trivi-2} we have the bound
 \begin{align}\label{est:I21}
    \Vert I_t^{2,1}\Vert_{L_\omega^1}\lesssim & n^{-p}.
 \end{align}

 Now we let
 \begin{align}\label{def:V-M}
   V_t:=V_t^2+V_t^3+V_t^4+I_t^{2,1},\quad M_t:=M_t^1+M_t^2,
 \end{align}
 then with \eqref{est:V1}, \eqref{est:I1} and \eqref{est:I2}  we can write
 \begin{align*}
   \phi_t\leq N\big(|x_0-x_0^n|^p+\int_{0}^{t}\phi_r\dd r+V_t\big)+M_t.
 \end{align*}
From \eqref{est:V2}, \eqref{est:V3}, \eqref{est:V4}, \eqref{est:V5} and \eqref{est:I21} we have the estimate
 \begin{align}\label{est:all-V}
 \Vert V_t\Vert_{L_\omega^1}\lesssim \big(n^{-\frac{1+\alpha}{2}+\epsilon}\big)^p.
 \end{align}
 { Notice that $\phi_t$ defined in \eqref{def:phi-X-Xn} and $V_t$ defined in \eqref{def:V-M} are both nonnegative nondecreasing processes, and $M_t$ from  \eqref{def:V-M} is $\mathcal{F}_t$-martingale. Therefore from an appropriate version of Gronwall's inequality (for an ``appropriate version'' see e.g \cite[{Lemma 3.8}]{LL} and  \eqref{est:all-V} and
the claimed bound \eqref{est:Milstein-bounds-Holder} follows.}
 \end{proof}

\appendix

\section{Proof for regularity estimates of PDEs }
\label{App:PDE}

\begin{proof}[Proof of \cref{thm:pde-Holder}]
 As for the existence and uniqueness, it was already shown in \cite[Theorem 4.3.1]{Kry}.
 Therefore it is enough to show \eqref{est:pde-Holder}. We start with the case that $a$ is a constant matrix and first order term of \eqref{PDE} vanishes.
  We stress that all proportionality constants in the relations $\lesssim$ below depend on $d,\alpha,\Vert b\Vert_{\mathcal{C}^\alpha},\lambda,$ and not on $\theta$.
  \begin{itemize}
    \item[Case I.] $a$ is a constant positive definite matrix and $b\equiv0$. 
     \end{itemize}
     As explained in \cite[Proof of Ch.1, Section 6, 2. Lemma]{Kry-Sobolev}, for a general non-degenerate constant matrix $a$, by changing of coordinate it is enough to consider the following resolvent equation
     \begin{align}\label{PDE:simple}
      \frac{1}{2}\Delta u-\theta u=f.
     \end{align}
    The solution to the above resolvent equation can be represented as
    \begin{align*}
      u=\int_{0}^{\infty} e^{-\theta t}P_{t}f\dd t.
    \end{align*}
     {  Firstly, for $\gamma,\alpha>0$, based on the  the fact that for $t\geq1$
    \begin{align*}
     \sup_{t\geq 1}\|P_{t}f\|_{\mathcal{C}^{\gamma}}  \lesssim  \|P_{1}f\|_{\mathcal{C}^{\gamma}}
     \lesssim \|f\|_{\mathcal{C}^{\alpha}},
    \end{align*}
    and for $t\in(0,1]$
    \begin{align*}
        \|P_{t}f\|_{\mathcal{C}^{\gamma}} \lesssim t^{-\frac{\gamma-\alpha}{2}}\|f\|_{\mathcal{C}^{\alpha}},
    \end{align*} together with $ e^{-\theta t} t^{\iota}\lesssim_\iota 1$ for any $\iota>0$ uniformly in $\theta\geq1$, $t>0$, we get for $t\geq 0$
    \begin{align}\label{eat:exp-ho}
         e^{-\theta t}\|P_{t}f\|_{\mathcal{C}^{\gamma}} \lesssim t^{-\frac{\gamma-\alpha}{2}}\|f\|_{\mathcal{C}^{\alpha}}.
    \end{align}
   So we have for $\gamma\in[\alpha,\alpha+2)$ and  $\theta\geq1$ 
    \begin{align}\label{est:pde-constant}
        \Vert u\Vert_{\mathcal{C}^{\gamma}}\leq& \big(\int_{0}^{1} +\int_{1}^{\infty}\big)e^{-\theta t}\|P_{t}f\|_{\mathcal{C}^{\gamma}}\dd t
     \nonumber\\ \lesssim& \int_{0}^{1} e^{-\theta t}t^{-\frac{\gamma-\alpha}{2}}\dd t\|f\|_{\mathcal{C}^{\alpha}}+\sup_{t\geq 1}\|P_{t}f\|_{\mathcal{C}^{\gamma}}\int_1^\infty e^{-\theta t}\dd t
 \nonumber\\ \lesssim&\theta^{\frac{\gamma-(2+\alpha)}{2}}\|f\|_{\mathcal{C}^{\alpha}}+\|f\|_{\mathcal{C}^{\alpha}}\theta^{-1} e^{-\theta}\lesssim \theta^{\frac{\gamma-(2+\alpha)}{2}}\|f\|_{\mathcal{C}^{\alpha}}.
    \end{align}}
    { For $\gamma=\alpha+2$, following the proof of \cite[Theorem 5.19]{AL}, for any $\xi>0$, we write $u=:\kappa_1^\xi+\kappa_2^\xi=:\int_{0}^{\xi} e^{-\theta t}P_{t}f\dd t+\int_{\xi}^{\infty} e^{-\theta t}P_{t}f\dd t$.
    From \eqref{eat:exp-ho} we have  for any $\nu\in(0,2)$
    \begin{align*}
      \|\kappa_1^\xi\|_{\mathcal{C}^{\nu+\alpha}}&\lesssim 
    \int_0^\xi t^{-\frac{\nu}{2}}\dd t\|f\|_{\mathcal{C}^\alpha}\lesssim \xi^{1-\frac{\nu}{2}}\|f\|_{\mathcal{C}^\alpha},\\
      \|\kappa_2^\xi\|_{\mathcal{C}^{2+\nu+\alpha}}&
      \lesssim\int^\infty_\xi t^{-1-\frac{\nu}{2}}\dd t\|f\|_{\mathcal{C}^\alpha}\lesssim\xi^{-\frac{\nu}{2}}\|f\|_{\mathcal{C}^\alpha}.
    \end{align*}
    By the K-method of interpolation  (see e.g. \cite[Example 5.15, Proof of Theorem 5.19 (p.149)]{AL}) we can conclude that
    \begin{align}
        \label{est:critical}\|u\|_{\mathcal{C}^{2+\alpha}}\lesssim\|f\|_{\mathcal{C}^\alpha}.
    \end{align}
    Hence by \eqref{est:pde-constant} and \eqref{est:critical} we can  conclude for $\gamma\in[\alpha,\alpha+2]$ and $\theta\geq1$
    \begin{align*}
        \Vert u\Vert_{\mathcal{C}^{\gamma}}\lesssim \theta^{\frac{\gamma-(2+\alpha)}{2}}\|f\|_{\mathcal{C}^{\alpha}}.
    \end{align*}
   }
    \begin{itemize}
    \item[Case II.]  $a$ satisfies \cref{ass-a} and $b\equiv0$.
      \end{itemize}
      Here we apply the frozen coefficient method. Let $\xi$  be a nonnegative smooth function such that $\xi\equiv 1$ near the origin  with   support in $B_\delta=\left\{x\in\mR^d:|x|\leq \delta\right\}$, where $\delta>0$ is to be chosen later. Define for $z\in\mR^d$
            \begin{align}\label{def:z}
        \xi_{z}(x):=\xi(x-z),\quad a_z:=a(z), \quad u_z(x):=\xi_{z}(x)u(x),\quad f_z(x):=\xi_{z}(x)f(x).
      \end{align}
      Then we can observe that
 \begin{align*}
    a^{ij}_z\partial_{ij}u_z- \theta u_z=g_z
 \end{align*}
 where
 \begin{align*}
     g_z:=&f_z-(a^{ij}\partial_{ij}u )\xi_{z}+a^{ij}_z\partial_{ij}u_z
     \\=&f_z-(a^{ij}-a^{ij}_z)\partial_{ij}u\cdot\xi_{z}+a^{ij}_z(\partial_iu\partial_j\xi_{z}
     +\partial_ju\partial_i\xi_{z}+u\partial_{ij}\xi_{z}).
 \end{align*}

  We use the fact that  $a$ satisfies \cref{ass-a} then get
       \begin{align*}
         \Vert g_z\Vert_{\mathcal{C}^\alpha}\lesssim  & \big(  \Vert f_z\Vert_{\mathcal{C}^\alpha} +\Vert a\Vert_{\mathcal{C}^\alpha} \delta^\alpha \Vert u_z\Vert_{\mathcal{C}^{2+\alpha}}+\lambda^{-1}(\Vert u_z\Vert_{\mathcal{C}^{1+\alpha}}\Vert\partial_i\xi_{z}\Vert_{\mathcal{C}^\alpha}
         +\Vert\partial_{ij}\xi_{z}\Vert_{\mathcal{C}^\alpha}\Vert u_z\Vert_{\mathcal{C}^\alpha})\big)
       \\\lesssim &\big(  \Vert f_z\Vert_{\mathcal{C}^\alpha} + \delta^{\alpha}\Vert u_z\Vert_{\mathcal{C}^{2+\alpha}}+\lambda^{-1}(\delta^{-(1+\alpha)}\Vert u_z\Vert_{\mathcal{C}^{1+\alpha}}+\delta^{-(2+\alpha)}\Vert u_z\Vert_{\mathcal{C}^{\alpha}})\big).
       \end{align*}
       By interpolation, we have that for any $\epsilon>0$ there exists $C_\epsilon>0$ so that
       \begin{align*}
         \Vert u_z\Vert_{\mathcal{C}^{1+\alpha}}\leq C_\epsilon   \Vert u_z\Vert_{\mathcal{C}^{\alpha}}+\epsilon  \Vert u_z\Vert_{\mathcal{C}^{2+\alpha}},
       \end{align*}
       it implies that
       \begin{align}\label{est:g}
      \Vert g_z\Vert_{\mathcal{C}^\alpha}\lesssim \big(  \Vert f_z\Vert_{\mathcal{C}^\alpha} + (\delta^{\alpha}+\delta^{-(1+\alpha)}\epsilon)\Vert u_z\Vert_{\mathcal{C}^{2+\alpha}}+(C_\epsilon\delta^{-(1+\alpha)}+\delta^{-(2+\alpha)})\Vert u_z\Vert_{\mathcal{C}^{\alpha}}\big).
       \end{align}
       First by taking $\gamma=2+\alpha$ {and applying \eqref{est:critical}
       then 
       \begin{align*}
          \Vert u_z\Vert_{\mathcal{C}^{2+\alpha}}&\leq C_1 
          \Vert g_z\Vert_{\mathcal{C}^\alpha}
          \\&\leq C_2 
          \big( \Vert f_z\Vert_{\mathcal{C}^\alpha} + (\delta^{\alpha}+\delta^{-(1+\alpha)}\epsilon)\Vert u_z\Vert_{\mathcal{C}^{2+\alpha}}+(C_\epsilon\delta^{-(1+\alpha)}+\delta^{-(2+\alpha)})\Vert u_z\Vert_{\mathcal{C}^{\alpha}}\big)
       \end{align*}
       with some $C_1,C_2\lesssim 1$.
     We first fix $\delta$ to be small enough and then $\epsilon$ to be small enough such that }$C_2\big(\delta^{\alpha}+\delta^{-(1+\alpha)}\epsilon\big)<\frac{1}{2}$. Then we get
        \begin{align}\label{est:u-C2alpha}
           \Vert u_z\Vert_{\mathcal{C}^{2+\alpha}}\lesssim \theta^{\frac{\alpha}{2}}\big(\Vert f_z\Vert_{\mathcal{C}^\alpha}+\Vert u_z\Vert_{\mathcal{C}^{\alpha}}\big).
       \end{align}
       Plug the above into \eqref{est:g} then get
       \begin{align}\label{est:g-u}
         \Vert g_z\Vert_{\mathcal{C}^\alpha}\lesssim \Vert f_z\Vert_{\mathcal{C}^\alpha}+\Vert u_z\Vert_{\mathcal{C}^{\alpha}}.
       \end{align}
       We again use \eqref{est:pde-constant}  by taking $\gamma=\alpha$, together with \eqref{est:g-u}
       \begin{align*}
          \Vert u_z\Vert_{\mathcal{C}^\alpha}\leq C_3\theta^{-1} \Vert g_z\Vert_{\mathcal{C}^\alpha}\leq C_4\theta^{-1}\big( \Vert f_z\Vert_{\mathcal{C}^\alpha}+\Vert u_z\Vert_{\mathcal{C}^{\alpha}}\big),
       \end{align*}
       again here the positive constants $C_3,C_4\lesssim1$. Take $\theta$ to be large so that $C_4\theta^{-1}\leq \frac{1}{2}$ then we have
       \begin{align}\label{est:u-alpha}
           \Vert u_z\Vert_{\mathcal{C}^\alpha}\lesssim \Vert f_z\Vert_{\mathcal{C}^\alpha}.
       \end{align}

      Combining  \eqref{est:pde-constant}, \eqref{est:g-u} and \eqref{est:u-alpha} yields
      \begin{align*}
         \Vert u_z\Vert_{\mathcal{C}^\gamma} \lesssim &  \theta^{\frac{\gamma-(2+\alpha)}{2}}\|g_z\|_{\mathcal{C}^{\alpha}}
         \lesssim\theta^{\frac{\gamma-(2+\alpha)}{2}}\|f_z\|_{\mathcal{C}^{\alpha}}.
      \end{align*}
      { We recall from \cite[Lemma 4.1.1]{Kry} the following holds (note that $\delta$ at this point is already fixed):
      \begin{align*}
         \Vert u\Vert_{\mathcal{C}^\gamma}\leq
         C(\gamma, d)\sup_z\Vert u_z\Vert_{\mathcal{C}^\gamma}.
      \end{align*}
      On the other hand,  by the elementary inequality $|u(x)v(x)-u(y)v(y)|\leq |u(x)||v(x)-v(y)|+ |v(y)||u(x)-u(y)|$ we have
      \begin{align*}
        \sup_z\Vert f_z\Vert_{\mathcal{C}^\alpha}\lesssim  \Vert f\Vert_{\mathcal{C}^\alpha}.
      \end{align*}
      }
      Therefore, taking supermum over $z\in\mR^d$ yields
      \begin{align}\label{est:pde-aij}
         \Vert u\Vert_{\mathcal{C}^\gamma} \lesssim\theta^{\frac{\gamma-(2+\alpha)}{2}}\|f\|_{\mathcal{C}^{\alpha}}.
      \end{align}
       \begin{itemize}
    \item[Case III.]  $a$ satisfies \cref{ass-a} and $b\in\mathcal{C}^{\alpha}$.
  \end{itemize}
  By \eqref{est:pde-aij}, for $\theta>\theta_0$  the following holds
  \begin{align}\label{est:aij-b}
    \Vert u\Vert_{\mathcal{C}^\gamma} \lesssim \theta^{\frac{\gamma-(2+\alpha)}{2}}\big(\|f\|_{\mathcal{C}^{\alpha}}+\|b\cdot\nabla u\|_{\mathcal{C}^{\alpha}}\big).
  \end{align}
 By taking $\gamma=1+\alpha$, it is evident that
  \begin{align*}
   \Vert u\Vert_{\mathcal{C}^{1+\alpha}} \leq& C \theta^{-\frac{1}{2}}\big(\|f\|_{\mathcal{C}^{\alpha}}+\|b\cdot\nabla u\|_{\mathcal{C}^{\alpha}}\big)
   \leq C\theta^{-\frac{1}{2}}\big(\|f\|_{\mathcal{C}^{\alpha}}+\|b \|_{\mathcal{C}^{\alpha}}\Vert u\Vert_{\mathcal{C}^{1+\alpha}}\big)
  \end{align*}
with some $C\lesssim 1$.  Now we take $\theta$ to be larger enough so that
  \begin{align*}
    C\theta^{-\frac{1}{2}}\|b \|_{\mathcal{C}^{\alpha}}\leq \frac{1}{2},
  \end{align*}
  we denote this $\theta$ as $\theta^*$ (depending on $\|b \|_{\mathcal{C}^{\alpha}},\alpha,\lambda$). Therefore for $\theta\geq\theta^*$, we get from \eqref{est:aij-b} that
  \begin{align*}
     \Vert u\Vert_{\mathcal{C}^\gamma} \lesssim \theta^{\frac{\gamma-(2+\alpha)}{2}}\|f\|_{\mathcal{C}^{\alpha}}.
  \end{align*}
  We get the desired estimate \eqref{est:pde-Holder}.
\end{proof}

\section*{Acknowledgements}
\noindent This research was funded in whole or in part by the Austrian Science Fund (FWF)    [10.55776/P34992]. For open access purposes, the author has applied a CC BY public copyright license to any author accepted manuscript version arising from this submission.

\end{document}